\documentclass[11pt]{article}

\usepackage[T1]{fontenc}
\usepackage{amssymb}
\usepackage{amsthm}
\usepackage{mathtools}
\usepackage{enumitem}
\usepackage{aliascnt}
\usepackage{booktabs}
\usepackage{placeins}
\usepackage{float}
\usepackage[hidelinks]{hyperref}
\numberwithin{equation}{section}
\newtheorem{theorem}{Theorem}[section]

\newaliascnt{proposition}{theorem}
\newtheorem{proposition}[proposition]{Proposition}
\aliascntresetthe{proposition}

\newaliascnt{lemma}{theorem}
\newtheorem{lemma}[lemma]{Lemma}
\aliascntresetthe{lemma}

\newaliascnt{corollary}{theorem}
\newtheorem{corollary}[corollary]{Corollary}
\aliascntresetthe{corollary}

\newaliascnt{conjecture}{theorem}

\aliascntresetthe{conjecture}

\newaliascnt{problem}{theorem}
\newtheorem{problem}[problem]{Problem}
\aliascntresetthe{problem}

\theoremstyle{definition}
\newaliascnt{definition}{theorem}
\newtheorem{definition}[definition]{Definition}
\aliascntresetthe{definition}

\newaliascnt{claim}{theorem}
\newtheorem{claim}[claim]{Claim}
\aliascntresetthe{claim}

\theoremstyle{remark}
\newaliascnt{remark}{theorem}

\aliascntresetthe{remark}

\theoremstyle{plain}

\usepackage[noabbrev,capitalize]{cleveref}

\crefname{theorem}{Theorem}{Theorems}
\crefname{proposition}{Proposition}{Propositions}
\crefname{lemma}{Lemma}{Lemmas}
\crefname{corollary}{Corollary}{Corollaries}
\crefname{conjecture}{Conjecture}{Conjectures}
\crefname{claim}{Claim}{Claims}
\crefname{definition}{Definition}{Definitions}
\crefname{remark}{Remark}{Remarks}
\crefname{problem}{Problem}{Problems}

\newcommand{\Z}{\mathbb{Z}}

\title{On a problem concerning $\Z_3$-Ramsey numbers}
\author{Cheng Chi$^\ast$
\and Jialin He$^\dagger$}
\date{}

\begin{document}

\maketitle

\begingroup
\renewcommand{\thefootnote}{\fnsymbol{footnote}}
\footnotetext[1]{School of Mathematical Sciences, Shanghai Jiao Tong University, 800 Dongchuan Road, Shanghai 200240, China.
  Email: chengchi@sjtu.edu.cn.
  Supported by the National Key R\&D Program of China under Grant No.~2022YFA1006400 and by the National Natural Science Foundation of China under Grant No.~12571376.}
\footnotetext[2]{School of Mathematical Sciences, Key Laboratory of MEA (Ministry of Education), Shanghai Key Laboratory of PMMP, and Nantong Institute for Applied Mathematics and Artificial Intelligence, East China Normal University, Shanghai 200241, China.
  Email: jlhe@math.ecnu.edu.cn.
  Supported in part by the Science and Technology Commission of Shanghai Municipality under Grant No.~22DZ2229014.}
\endgroup

\begin{abstract}
  For a graph $G$ with $3\mid e(G)$, let $R(G,\Z_3)$ denote the least integer $N$ such that every edge weighting $w:E(K_N)\to\Z_3$ admits a copy of $G$ whose edge weights sum to zero.
  We prove that every $n$-vertex graph $G$ with $n\ge6$ and $3\mid e(G)$ satisfies $R(G,\Z_3)\le n+5$.
  Moreover, the stronger bound $R(G,\Z_3)\le n+4$ holds whenever $G$ contains an induced $P_4$ or an induced $2K_2$.
  Consequently, we resolve a problem posed by Caro and Mifsud: every graph $G$ with $3\mid e(G)$ satisfies $R(G,\Z_3)\le |V(G)|+8$, with equality if and only if $G\cong K_3$.
\end{abstract}

\section{Introduction}

Zero-sum Ramsey theory, introduced by Bialostocki and Dierker \cite{BialostockiDierker1990theorems,BialostockiDierker1990}, is an algebraic variant of classical Ramsey theory motivated by the Erd\H{o}s--Ginzburg--Ziv theorem; see also Caro's survey~\cite{Caro1996}.
In classical Ramsey theory, edge labels are regarded merely as colors, and one seeks a copy of a prescribed graph whose edges all have the same label.
In the zero-sum setting, the labels are elements of a finite abelian group $\Gamma$, and monochromaticity is replaced by an algebraic balance condition.
More precisely, for a graph $G$ satisfying $|\Gamma|\mid e(G)$, the zero-sum Ramsey number $R(G,\Gamma)$ is the least integer $N$ such that every weighting $w:E(K_N)\to\Gamma$ contains a copy $G'$ of $G$ for which
\[
  \sum_{e\in E(G')}w(e)=0 \quad\text{in }\Gamma.
\]
Every monochromatic copy of $G$ is zero-sum under this divisibility condition, so the ordinary multicolor Ramsey theorem guarantees that $R(G,\Gamma)$ is finite.

The search for strong general upper bounds is also guided by a familiar principle from ordinary two-color Ramsey theory: if $|V(G)|=n$, then $r(G,G)\le r(K_n,K_n)$, where $r$ denotes the ordinary two-color Ramsey number.
This makes complete graphs a natural first benchmark.
An early systematic formulation of the corresponding linear---in fact, additive---program in zero-sum Ramsey theory is due to Caro~\cite{Caro1994}.
He proved that, for every fixed $k$ and $r$, there is a constant $c(k,r)$ such that
\[
  R(K_n^{(r)},\Gamma)\le n+c(k,r)
\]
for every group $\Gamma$ of order $k$ whenever $k\mid\binom{n}{r}$, where $K_n^{(r)}$ denotes the complete $r$-uniform hypergraph.
The analogy with ordinary Ramsey theory has an important limitation, since zero-sum Ramsey numbers are not monotone under taking subgraphs.
In particular, a zero-sum copy of $K_n$ need not contain a zero-sum copy of an $n$-vertex graph $G$.
Hence results for complete graphs provide the natural first step, but not an automatic reduction of the general problem.

Caro~\cite{CaroPersonalCommunication2026} suggested the following broader additive question: for each finite abelian group $\Gamma$, does there exist a constant $f(\Gamma)$ such that $R(G,\Gamma)\le |V(G)|+f(\Gamma)$ for every graph $G$ satisfying $|\Gamma|\mid e(G)$?
A parallel question can be asked for $r$-uniform hypergraphs, with the constant also allowed to depend on $r$.
The case $\Gamma=\Z_2$ provides the first compelling evidence for such an additive principle.
Alon and Caro~\cite{AlonCaro1993} proved that every $n$-vertex graph $G$ with an even number of edges satisfies $R(G,\Z_2)\le n+2$, and Caro later determined $R(G,\Z_2)$ exactly for all such graphs~\cite{Caro1994binary}.
Thus $\Z_3$ is the next fundamental case.
The complete-graph case over $\Z_3$ was developed through a sequence of exact results.
Work of Bialostocki and Dierker, Harborth and Piepmeyer, and Caro determined $R(K_n,\Z_3)$ for every $n$ satisfying $3\mid\binom n2$, except when $n\equiv7\pmod9$ and $n\ge16$~\cite{BialostockiDierker1990,HarborthPiepmeyer1994,HarborthPiepmeyer1996,Caro1997}.
Most recently, Chi, He, and Ma~\cite{ChiHeMaCompleteGraphs2026} proved that $R(K_n,\Z_3)=n+3$ for every $n\ge10$ with $n\equiv1\pmod3$, thereby completing the determination for complete graphs over $\Z_3$.

Beyond complete graphs, exact or sharp additive results are known for several natural graph classes.
Caro and Mifsud~\cite{CaroMifsud2025} proved that every $n$-vertex forest $F$ with $3\mid e(F)$ satisfies $R(F,\Z_3)\le n+2$ and determined exact values for several infinite families of trees.
Alvarado, Colucci, and Parente~\cite{AlvaradoColucciParente2025} subsequently determined $R(F,\Z_3)$ for every forest $F$ without isolated vertices; the value is always one of $n$, $n+1$, and $n+2$.
Chi and He obtained exact results over $\Z_3$ for cycles, wheels, and complete bipartite graphs~\cite{ChiHeCyclesWheels,ChiHeCompleteBipartite2026}.
Costa~\cite{CostaFriendship2026} determined exact zero-sum Ramsey numbers for friendship graphs, books, and triangular windmills.
At a broader level, linear bounds are known for bounded-degree graphs over arbitrary finite abelian groups~\cite{KatzLianMalekshahianShapiro2025}, while additive bounds over prime cyclic groups have been obtained for suitable degenerate graphs~\cite{Shapiro2026} and graphs admitting a sufficiently large $2$-packing~\cite{HeathSimmons2026}.
Taken together, these results provide substantial evidence for a uniform additive bound over $\Z_3$, but they all impose structural restrictions on the target graph and therefore leave the general case open.
The known value $R(K_3,\Z_3)=11$~\cite{ChungGraham1983} shows that the additive constant in any universal bound cannot be smaller than eight.
It is therefore natural to ask whether eight always suffices and whether equality can only occur for $K_3$.
Caro and Mifsud posed the following problem.
\begin{problem}[Caro and Mifsud \cite{CaroMifsud2025}]
Let $G$ be an $n$-vertex graph with $3\mid e(G)$.
Is it true that
\[
  R(G,\Z_3)\le n+8,
\]
with equality if and only if $G\cong K_3$?
\end{problem}

We answer this problem affirmatively.
More strongly, we prove that every graph $G$ with at least six vertices and $3\mid e(G)$ satisfies $R(G,\Z_3)\le |V(G)|+5$.

\begin{theorem}
  \label{thm:main}
  Let $n\ge6$, and let $G$ be an $n$-vertex graph with $3\mid e(G)$.
  Then
  \[
    R(G,\Z_3)\le
    \begin{cases}
      n+4, & \text{if $G$ contains an induced $P_4$ or an induced $2K_2$}, \\
      n+5, & \text{otherwise.}
    \end{cases}
  \]
\end{theorem}

The proof also uses the exact values
\[
  R(K_4,\Z_3)=7
  \quad\text{and}\quad
  R(K_6,\Z_3)=10,
\]
both due to Harborth and Piepmeyer~\cite{HarborthPiepmeyer1994}.
The remaining cases with $|V(G)|\le5$ are handled in \cref{sec:proof-cor-main}.
Together with the known value $R(K_3,\Z_3)=11$~\cite{ChungGraham1983}, this yields the following corollary.

\begin{corollary}\label{cor:main}
  Every graph $G$ with $3\mid e(G)$ satisfies $R(G,\Z_3)\le |V(G)|+8$, with equality if and only if $G\cong K_3$.
\end{corollary}

The proof of \cref{thm:main} has two branches.
If $G$ contains an induced $P_4$ or $2K_2$, then two independent ways of switching vertices within a copy of $G$ impose strong restrictions on any counterexample.
These restrictions force the host weighting to have an affine form on a large set of vertices, where a degree-weighted zero-sum assignment produces a zero-sum copy of $G$.
If $G$ contains neither an induced $P_4$ nor an induced $2K_2$, then its rigid structure permits repeated reductions by true and false twins.
After these reductions, either the same switching argument applies or a rooted-forest representation gives the irreducible graph in one of a few structural forms whose Ramsey bounds can be proved directly.

The remainder of the paper is organized as follows.
In \cref{sec:preliminaries}, we develop the weighting, switching, and structural tools used in both branches.
We prove the two cases of \cref{thm:main} in \cref{sec:proof-main,sec:proof-main2}.
In \cref{sec:proof-main2}, we reduce the complementary class to irreducible graphs and derive the structural forms needed for the Ramsey argument.
We then establish the Ramsey bounds for these forms in \cref{sec:main2-irreducible-bounds}, with the proofs for the exceptional graphs $K_5-e$ and $K_{2,2,2}$ given in Appendices~\ref{app:k5minus-bound} and~\ref{app:k222-bound}, respectively.
We derive \cref{cor:main} in \cref{sec:proof-cor-main}.

\section{Preliminaries}
\label{sec:preliminaries}

For disjoint graphs $X$ and $Y$, the notation $X\cup Y$ denotes their disjoint union, and $X+Y$ denotes the graph obtained from $X\cup Y$ by adding every edge with one endpoint in $V(X)$ and the other in $V(Y)$.
For a family $\mathcal F$ of graphs, a graph is \emph{induced $\mathcal F$-free} if it contains no induced copy of any graph in $\mathcal F$.
For a positive integer $m$, write $[m]=\{1,\ldots,m\}$.
Write $\Z_3^\times=\{1,2\}$ for the set of nonzero elements of $\Z_3$.
If $w:E(K)\to\Z_3$ is a weighting of a complete graph $K$ and $J\subseteq K$ is a subgraph, write
\[
  w(J)=\sum_{e\in E(J)}w(e).
\]

We shall use the Erd\H{o}s--Ginzburg--Ziv theorem to prove \cref{lem:degree-completion}.
\begin{theorem}[Erd\H{o}s, Ginzburg and Ziv \cite{ErdosGinzburgZiv1961}]
  \label{thm:EGZ}
  For every positive integer $m$, every list of $2m-1$ elements of $\Z_m$ contains $m$ terms whose sum is zero in $\Z_m$.
\end{theorem}

The next lemma is the bridge from affine edge weights to a zero-sum copy.
It assigns the vertices of $G$ to $n+2$ prescribed labels so that their degree-weighted contribution vanishes.
During the revision of this manuscript, we became aware that Costa had independently proved \cref{lem:degree-completion}; see \cite[Lemma~2.1]{CostaBoundedTreewidth2026}.

\begin{lemma}
  \label{lem:degree-completion}
  Let $G$ be an $n$-vertex graph with $3\mid e(G)$.
  Given labels $z_1,\ldots,z_{n+2}\in\Z_3$, there is an injection $\sigma:V(G)\hookrightarrow\{1,\ldots,n+2\}$ such that $\sum_{v\in V(G)}\deg_G(v)\cdot z_{\sigma(v)}=0$ in $\Z_3$.
\end{lemma}

\begin{proof}
  Let $U_1$ and $U_2$ be the sets of vertices whose degrees are congruent to $1$ and $2$ modulo three, respectively, and write $r=|U_1|$ and $s=|U_2|$.
  Since $r+2s\equiv\sum_{v\in V(G)}\deg_G(v)=2e(G)\equiv0\pmod3$, we have $r\equiv s\pmod3$.
  Choose a set $I\subseteq\{1,\ldots,n+2\}$ of size $r+s+2$.
  We first assign distinct indices from $I$ to the vertices in $U_1\cup U_2$.
  As long as one of $U_1$ and $U_2$ has at least three unassigned vertices, choose three such vertices from the same set.
  At each such step, the number of unused indices in $I$ is the number of unassigned vertices plus two and is therefore at least five.
  Choose any five unused indices.
  By \cref{thm:EGZ}, three of the corresponding labels have sum zero.
  Assign those indices to the chosen vertices.
  Their total contribution is zero, whether their degree residue is $1$ or $2$.
  Repeat this procedure until fewer than three vertices remain in each of $U_1$ and $U_2$.
  If $r'$ and $s'$ are the respective numbers of remaining vertices, then $0\le r',s'\le2$ and $r'\equiv s'\pmod3$.
  Thus, $r'=s'=a$ for some $a\in\{0,1,2\}$.

  If $a=0$, there is nothing more to assign.
  If $a=1$, four indices remain, so two of the corresponding labels are equal.
  Let $z$ be their common label, and assign those two indices to the remaining vertices of $U_1$ and $U_2$, respectively.
  Their total contribution is $z+2z=0$.
  Finally, suppose that $a=2$, so six labels remain.
  If some value $x$ occurs at least four times, choose two disjoint pairs of labels $(x,x)$.
  Otherwise, two distinct values $x$ and $y$ each occur at least twice, and we choose two disjoint pairs $(x,y)$.
  Thus, in either case, there are two disjoint pairs with the same label sum.
  Assign one pair to the two remaining vertices of $U_1$ and the other pair to the two remaining vertices of $U_2$.
  If their common sum is $z$, their total contribution is $z+2z=0$.
  We have therefore constructed an injection $\sigma$ on $U_1\cup U_2$ such that
  \[
    \sum_{v\in U_1\cup U_2}\deg_G(v) \cdot z_{\sigma(v)}=0.
  \]
  Since $\sigma$ uses $r+s$ of the $n+2$ label indices, $n-r-s+2$ indices remain unused.
  Hence, $\sigma$ can be extended injectively to the remaining $n-r-s$ vertices.
  Their degrees are divisible by three, so this extension does not change the weighted sum.
\end{proof}

For a finite set $V$, write $K_V$ for the complete graph on vertex set $V$.
Let $w:E(K_V)\to\Z_3$ be a weighting.
For four distinct vertices $a,b,u,v\in V$, define the ordered rectangle $(a,b;u,v)$ and its value by
\[
  \Delta_w(a,b;u,v)
  =w(au)+w(bv)-w(av)-w(bu).
\]
Note that interchanging the two vertices within either pair negates its value and therefore preserves whether it is zero.
We call $(a,b;u,v)$ \emph{zero} or \emph{nonzero} according as $\Delta_w(a,b;u,v)$ is zero or nonzero in $\Z_3$.
A weighting $w$ on $K_U$ is \emph{affine} if there are $c\in\Z_3$ and labels $\phi_u\in\Z_3$ for $u\in U$ such that $w(uv)=\phi_u+\phi_v+c$ for all distinct $u,v\in U$.

The next lemma is the $\Z_3$-case of \cite[Lemma~4.2]{ChiHeCyclesWheels}.
It turns the local vanishing of every rectangle difference into the global affine representation needed for the final embedding argument.

\begin{lemma}[Chi and He \cite{ChiHeCyclesWheels}]
  \label{lem:rectangle-form}
  Let $U$ be a set with $|U|\ge4$.
  If every ordered rectangle $(a,b;u,v)$ on four distinct vertices of $U$ is zero, then the weighting on $K_U$ is affine.
\end{lemma}

The following result extends an affine weighting from a core to all but at most two vertices when no two nonzero rectangles are vertex-disjoint.

\begin{lemma}
  \label{lem:compression}
  Let $V$ be a finite set and let $U\subseteq V$ with $|U|\ge6$.
  Let $w:E(K_V)\to\Z_3$ satisfy the following conditions:
  \begin{enumerate}[label=\textup{(\roman*)}]
    \item the restriction of $w$ to $K_U$ is affine;
    \item $K_V$ contains no two vertex-disjoint nonzero rectangles.
  \end{enumerate}
  Then there is a set $D\subseteq V$ with $|D|\le2$ such that the restriction of $w$ to $K_{V\setminus D}$ is affine.
\end{lemma}

\begin{proof}
  Set $B=V\setminus U$.
  Choose $c\in\Z_3$ and labels $\phi_u\in\Z_3$ for $u\in U$ such that $w(uv)=\phi_u+\phi_v+c$ for all distinct $u,v\in U$.
  For each $x\in B$, define a function $g_x:U\to\Z_3$ by $g_x(u)=w(xu)-\phi_u-c$.
  Call $x$ good if $g_x$ is constant, and write $g_x\equiv\lambda_x$ when $x$ is good.
  Call $x$ bad if $g_x$ is nonconstant.
  For distinct $a,u,v\in U$, a direct calculation gives
  \[
    \Delta_w(x,a;u,v)=g_x(u)-g_x(v).
  \]
  Thus, if $g_x(u)\ne g_x(v)$ for a bad vertex $x$, then $(x,a;u,v)$ is a nonzero rectangle for every $a\in U\setminus\{u,v\}$.
  We split according to the number of bad vertices in $B$.

  \medskip
  \noindent\textbf{Case 1. No vertex in $B$ is bad.}
  \medskip

  Define a graph $J$ on $B$ by joining $x$ and $y$ when $w(xy)\ne\lambda_x+\lambda_y+c$.
  If $J$ contained disjoint edges $x_1y_1$ and $x_2y_2$, choose distinct $r_1,s_1,r_2,s_2\in U$.
  For $i\in\{1,2\}$,
  \[
    \Delta_w(x_i,r_i;y_i,s_i)
    =w(x_i y_i)-\lambda_{x_i}-\lambda_{y_i}-c\ne0,
  \]
  giving two vertex-disjoint nonzero rectangles.
  Hence, the matching number of $J$ is at most one.
  Set $U_0=U$.
  If $J$ is empty, set $B_0=B$; otherwise, fix $xy\in E(J)$ and set $B_0=B\setminus\{x,y\}$.
  Every edge of $J$ meets $\{x,y\}$, so $J[B_0]$ is empty and $|V\setminus(U_0\cup B_0)|\le2$.

  \medskip
  \noindent\textbf{Case 2. Exactly one vertex $x\in B$ is bad.}
  \medskip

  Choose $u,v\in U$ with $g_x(u)\ne g_x(v)$ and $a\in U\setminus\{u,v\}$.
  If good vertices $y,z\in B$ satisfied $w(yz)\ne\lambda_y+\lambda_z+c$, choose distinct $r,s\in U\setminus\{a,u,v\}$.
  Then $\Delta_w(x,a;u,v)=g_x(u)-g_x(v)$ is nonzero, while $\Delta_w(y,r;z,s)=w(yz)-\lambda_y-\lambda_z-c\ne0$.
  These two rectangles are vertex-disjoint, a contradiction.
  Thus, all edges among the good vertices have the affine form.
  Set $U_0=U$ and $B_0=B\setminus\{x\}$.

  \medskip
  \noindent\textbf{Case 3. At least two vertices in $B$ are bad.}
  \medskip

  For a bad vertex $z$, let $\mathcal D_z=\bigl\{\{u,v\}\in\tbinom{U}{2}:g_z(u)\ne g_z(v)\bigr\}$.
  For distinct bad vertices $x,y$, every member of $\mathcal D_x$ meets every member of $\mathcal D_y$.
  Otherwise, disjoint pairs in $\mathcal D_x$ and $\mathcal D_y$, together with two further vertices of $U$, give two vertex-disjoint nonzero rectangles.

  We next show that there exists $a\in U$ such that both $g_x$ and $g_y$ are constant on $U\setminus\{a\}$.
  Choose $\{a,b\}\in\mathcal D_x$.
  No pair contained in $U\setminus\{a,b\}$ can belong to $\mathcal D_y$, since it would be disjoint from $\{a,b\}$.
  Hence, $g_y$ is constant, say equal to $\alpha$, on $U\setminus\{a,b\}$.
  Since $g_y$ is nonconstant, at least one of $g_y(a)$ and $g_y(b)$ differs from $\alpha$.
  Suppose that both differ from $\alpha$.
  For any $E\in\mathcal D_x$, choose $t\notin E\cup\{a,b\}$.
  Then $\{a,t\},\{b,t\}\in\mathcal D_y$, so $E$ must meet both pairs.
  As $t\notin E$, this forces $E=\{a,b\}$.
  However, for any $q\notin\{a,b\}$, at least one of $\{a,q\}$ and $\{b,q\}$ lies in $\mathcal D_x$, a contradiction.
  Thus, after interchanging $a$ and $b$ if necessary, $g_y(a)\ne\alpha=g_y(b)$, and consequently $g_y$ is constant on $U\setminus\{a\}$.
  If some $E\in\mathcal D_x$ did not contain $a$, choose $t\notin E\cup\{a\}$.
  Then $\{a,t\}\in\mathcal D_y$ would be disjoint from $E$, a contradiction.
  Therefore, every pair in $\mathcal D_x$ contains $a$, which means that $g_x$ is constant on $U\setminus\{a\}$.
  Thus, both $g_x$ and $g_y$ are constant on $U\setminus\{a\}$.

  We next show that such an exceptional vertex is unique for a nonconstant function.
  Suppose that $g_z$ is constant on both $U\setminus\{a\}$ and $U\setminus\{b\}$, where $a\ne b$, and choose $t\in U\setminus\{a,b\}$.
  Since $b,t\in U\setminus\{a\}$, we have $g_z(b)=g_z(t)$; similarly, since $a,t\in U\setminus\{b\}$, we have $g_z(a)=g_z(t)$.
  Thus, $g_z$ is constant on all of $U$, a contradiction.
  Now fix a bad vertex $x_0$.
  For each other bad vertex $y$, the conclusion for the pair $x_0,y$ gives a vertex $a_y\in U$ such that both $g_{x_0}$ and $g_y$ are constant on $U\setminus\{a_y\}$.
  Since $g_{x_0}$ is nonconstant, the uniqueness just proved shows that all the vertices $a_y$ are equal.
  Denote their common value by $u_0$.
  Then every bad vertex $z$ has $g_z$ constant on $U\setminus\{u_0\}$.

  Let $U'=U\setminus\{u_0\}$.
  By the choice of $u_0$, every bad vertex $z\in B$ has $g_z$ constant on $U'$, and the same is true for every good vertex by definition.
  For every $z\in B$, define $\lambda_z\in\Z_3$ by $g_z(u)=\lambda_z$ for all $u\in U'$.
  Define a defect graph $J$ on $B$ by joining $y$ and $z$ when $w(yz)\ne\lambda_y+\lambda_z+c$.
  Fix an arbitrary bad vertex $x\in B$.
  If an edge $yz\in E(J)$ were not incident with $x$, choose distinct $u,a,r,s\in U'$.
  Since $g_x(u_0)\ne\lambda_x$, the rectangles $(x,a;u_0,u)$ and $(y,r;z,s)$ would be vertex-disjoint and nonzero; for the second rectangle, $\Delta_w(y,r;z,s)=w(yz)-\lambda_y-\lambda_z-c\ne0$.
  Repeating this argument for each bad vertex $x\in B$, we conclude that every edge of $J$ is incident with every bad vertex.
  If at least three vertices are bad, set $U_0=U'$ and $B_0=B$; then $J$ has no edge.
  If exactly two vertices $x_1,x_2$ are bad, set $U_0=U'$ and $B_0=B\setminus\{x_1\}$, then $E(J)\subseteq\{x_1x_2\}$, so $J[B_0]$ has no edge.
  In both cases, $|V\setminus(U_0\cup B_0)|\le2$.
  \medskip

  In all three cases, every $g_z$ with $z\in B_0$ is constant on $U_0$.
  Denote its value by $\lambda_z$.
  Moreover, $w(yz)=\lambda_y+\lambda_z+c$ for all distinct $y,z\in B_0$.
  Let $D=V\setminus(U_0\cup B_0)$.
  For $v\in U_0\cup B_0$, set $\theta_v=\phi_v$ if $v\in U_0$ and $\theta_v=\lambda_v$ if $v\in B_0$.
  Then every edge $xy$ of $K_{U_0\cup B_0}$ satisfies $w(xy)=\theta_x+\theta_y+c$.
  Hence, $w$ is affine on $K_{V\setminus D}$ and $|D|\le2$, completing the proof of~\cref{lem:compression}.
\end{proof}

\section{Proof of \cref{thm:main}: When $G$ contains an induced $P_4$ or an induced $2K_2$}
\label{sec:proof-main}

Assume that $G$ contains an induced $P_4$ or an induced $2K_2$.
We first identify the pair of independent switches supplied by this configuration.

\begin{claim}
  \label{clm:private-neighbors}
  There are four distinct vertices $x,y,x_0,y_0\in V(G)$ such that $xy\notin E(G)$, $x_0\in N_G(x)\setminus N_G(y)$, and $y_0\in N_G(y)\setminus N_G(x)$.
\end{claim}

\begin{proof}
  If $G$ contains an induced path $v_1v_2v_3v_4$, take $(x,x_0,y_0,y)=(v_1,v_2,v_3,v_4)$.
  Otherwise, take $xx_0$ and $yy_0$ to be the two edges of an induced $2K_2$.
  In either case, $xx_0,yy_0\in E(G)$, whereas $xy,x_0y,xy_0\notin E(G)$, and the claim follows.
\end{proof}

Fix $x,y,x_0,y_0$ as in \cref{clm:private-neighbors}.
Let $w:E(K_{n+4})\to\Z_3$ be a weighting, and suppose for a contradiction that it contains no zero-sum copy of $G$.

\begin{claim}
  \label{clm:no-disjoint-nonzero-rectangles}
  The weighting $w$ contains no two vertex-disjoint nonzero rectangles.
\end{claim}

\begin{proof}
  For $i\in\{1,2\}$, suppose that $(a_i,b_i;u_i,v_i)$ are two vertex-disjoint nonzero rectangles.
  Set $(p_1,q_1)=(x,x_0)$ and $(p_2,q_2)=(y,y_0)$.
  Thus, for $i\in\{1,2\}$, $q_i$ is a private neighbor of $p_i$, while $p_1p_2\notin E(G)$.
  The vertices $p_i$ will be used for the two independent switches.

  Let $K=V(K_{n+4})\setminus\{a_1,b_1,u_1,v_1,a_2,b_2,u_2,v_2\}$.
  Both $K$ and $V(G)\setminus\{p_1,q_1,p_2,q_2\}$ have size $n-4$.
  Fix a bijection $\psi:V(G)\setminus\{p_1,q_1,p_2,q_2\}\to K$.
  By \cref{clm:private-neighbors}, $N_G(p_i)\setminus\{q_i\}\subseteq V(G)\setminus\{p_1,q_1,p_2,q_2\}$ for $i\in\{1,2\}$.
  Thus, for each $i\in\{1,2\}$ and $t\in\{u_i,v_i\}$, we may define
  \[
    \Delta_i(t)=
    \sum_{z\in\psi(N_G(p_i)\setminus\{q_i\})\cup\{t\}}
    \bigl(w(b_i z)-w(a_i z)\bigr).
  \]
  A direct calculation gives, for $i\in\{1,2\}$, $\Delta_i(u_i)-\Delta_i(v_i)=-\Delta_w(a_i,b_i;u_i,v_i)\ne0$.
  Hence, for each $i$, we can choose $t_i\in\{u_i,v_i\}$ such that $\Delta_i(t_i)\ne0$, and set $\delta_i=\Delta_i(t_i)$.
  Define an embedding $\sigma_0:V(G)\to V(K_{n+4})$ by $\sigma_0(p_i)=a_i$, $\sigma_0(q_i)=t_i$ for $i=1,2$ and $\sigma_0(z)=\psi(z)$ for $z\notin\{p_1,q_1,p_2,q_2\}$.
  Since the two rectangles are vertex-disjoint and $\psi$ maps into $K$, $\sigma_0$ is injective.
  Switching the image of $p_i$ from $a_i$ to $b_i$ changes the copy weight by $\delta_i$.
  Since $p_1p_2=xy\notin E(G)$, the two switches are independent.
  Let $z_0$ be the weight of the initial copy.
  For $\varepsilon_1,\varepsilon_2\in\{0,1\}$, perform the $i$-th switch precisely when $\varepsilon_i=1$.
  The resulting copy has weight $z_0+\varepsilon_1\delta_1+\varepsilon_2\delta_2$.
  Since $\delta_1,\delta_2\ne0$, the set of possible increments is $\{0,\delta_1,\delta_2,\delta_1+\delta_2\}=\Z_3$.
  Hence, one of these increments equals $-z_0$, and the corresponding copy is zero-sum, a contradiction.
\end{proof}

We now find an affine set of at least $n+2$ host vertices.
If every rectangle in $K_{n+4}$ is zero, then \cref{lem:rectangle-form} shows that $w$ is affine on all $n+4$ vertices.
Otherwise, choose a nonzero rectangle $(a,b;u,v)$ and set $U=V(K_{n+4})\setminus\{a,b,u,v\}$.
By \cref{clm:no-disjoint-nonzero-rectangles}, every rectangle on $U$ is vertex-disjoint from $(a,b;u,v)$ and hence is zero.
Since $|U|=n\ge6$, \cref{lem:rectangle-form} shows that $w$ is affine on $K_U$, and then \cref{lem:compression} gives a set $D$ with $|D|\le2$ such that $w$ is affine on $K_{V(K_{n+4})\setminus D}$.
In either case, $w$ is affine on an $(n+2)$-vertex set $C$.
We choose $c\in\Z_3$ and labels $\phi_u\in\Z_3$, for $u\in C$, such that $w(uv)=\phi_u+\phi_v+c$ for all distinct $u,v\in C$.
By \cref{lem:degree-completion}, there is an embedding $\sigma:V(G)\hookrightarrow C$ satisfying $\sum_{v\in V(G)}\deg_G(v)\cdot\phi_{\sigma(v)}=0$.
For the copy $G_\sigma$ determined by $\sigma$, we have
\[
  \begin{aligned}
    w(G_\sigma)
     & =\sum_{xy\in E(G)}w\bigl(\sigma(x)\sigma(y)\bigr)              \\
     & =\sum_{v\in V(G)}\deg_G(v)\cdot\phi_{\sigma(v)}+c\cdot e(G)=0.
  \end{aligned}
\]
Here the last equality follows from $3\mid e(G)$.
Thus, $G_\sigma$ is a zero-sum copy of $G$, contradicting our assumption.
Therefore, $R(G,\Z_3)\le n+4$ whenever $G$ contains an induced $P_4$ or an induced $2K_2$.

\section{Proof of \cref{thm:main}: When $G$ contains neither an induced $P_4$ nor an induced $2K_2$}
\label{sec:proof-main2}

Assume that $G$ contains neither an induced $P_4$ nor an induced $2K_2$.
The proof has three stages.
First, we delete suitable blocks of twins and show that a zero-sum copy of the reduced graph can be lifted back to $G$.
Second, if the reduced graph contains an \textup{(SVA)} configuration (see \cref{def:SVA}), two independent switches give the desired $n+4$ bound.
Finally, when no such configuration exists, a rooted-forest representation reduces the problem to a small family of structural forms.
For convenience, we use the convention $R(K_0,\Z_3)=0$.

\subsection{Reduction to irreducible graphs}
\label{sec:main2-reduction}

The following definition identifies the twin blocks that will be deleted; their sizes are chosen so that their edge contributions can be recovered with total weight zero in $\Z_3$.

\begin{definition}
  Distinct nonadjacent vertices $u$ and $v$ of a graph $Q$ are \emph{false twins} if $N_Q(u)=N_Q(v)$, whereas distinct adjacent vertices $u$ and $v$ are \emph{true twins} if $N_Q[u]=N_Q[v]$.
  A \emph{false-three reduction} deletes a set of three pairwise false-twin vertices, and a \emph{true-six reduction} deletes a set of six pairwise true-twin vertices.
  A graph is \emph{$3/6$-twin-reduced} if it contains neither a set of three pairwise false-twin vertices nor a set of six pairwise true-twin vertices.
\end{definition}

The following lifting proposition preserves the edge-count congruence and transfers an additive Ramsey bound from the reduced graph back to the original one.

\begin{proposition}
  \label{prop:twin-lifting}
  Let $Q=Q_m\supset Q_{m-1}\supset\cdots\supset Q_0=H$ be a sequence in which $Q_{j-1}$ is obtained from $Q_j$ by a false-three or true-six reduction for each $j\in\{1,\ldots,m\}$.
  Then $e(Q)\equiv e(H)\pmod3$.
  If $3\mid e(Q)$, then moreover
  \[
    R(Q,\Z_3)\le
    \max\bigl\{R(H,\Z_3),\,|V(Q)|+4\bigr\}.
  \]
\end{proposition}

\begin{proof}
  We first check the edge-count congruence.
  For a false-three reduction $J\to L$, let $A\subseteq V(L)$ be the common neighborhood of the three deleted vertices.
  Then $e(J)-e(L)=3|A|$.
  For a true-six reduction, the six deleted vertices form a clique and have a common neighborhood $A\subseteq V(L)$, so $e(J)-e(L)=6|A|+\binom62$.
  Both differences are divisible by three.
  Summing over the reduction sequence gives $e(Q)\equiv e(H)\pmod3$.

  We now lift a zero-sum copy of $H$ through the reduction sequence.
  Assume that $3\mid e(Q)$.
  The congruence $e(Q)\equiv e(H)\pmod3$ implies $3\mid e(H)$.
  Set $N=\max\{R(H,\Z_3),|V(Q)|+4\}$ and consider an arbitrary weighting $w:E(K_N)\to\Z_3$.
  By the definition of $R(H,\Z_3)$, there is an embedding $\sigma_0:V(H)\to V(K_N)$ whose image has weight zero.
  Suppose inductively that $\sigma_{j-1}$ is a zero-sum embedding of $Q_{j-1}$.
  We construct a zero-sum extension $\sigma_j$ to $Q_j$.

  First suppose that $Q_j$ is obtained by adjoining three false twins with common neighborhood $A\subseteq V(Q_{j-1})$.
  The set $U$ of unused host vertices satisfies
  \[
    |U|=N-|V(Q_{j-1})|
    \ge |V(Q)|+4-|V(Q_{j-1})|
    \ge |V(Q_j)|+4-|V(Q_{j-1})|=7.
  \]
  For $z\in U$, define $\lambda(z)=\sum_{a\in A}w\bigl(z\sigma_{j-1}(a)\bigr)$.
  Choose any five unused vertices.
  By \cref{thm:EGZ}, three of their $\lambda$-values sum to zero.
  Mapping the false twins to the corresponding vertices adds no internal edges and adds total weight zero on the edges to $A$.
  This gives the required embedding $\sigma_j$.

  Suppose instead that $Q_j$ is obtained by adjoining six true twins.
  With $A$ and $U$ defined analogously,
  \[
    |U|\ge |V(Q)|+4-|V(Q_{j-1})|
    \ge |V(Q_j)|+4-|V(Q_{j-1})|=10.
  \]
  For $z\in U$, define $\lambda(z)=\sum_{a\in A}w\bigl(z\sigma_{j-1}(a)\bigr)$ and, on $K_U$, define the auxiliary weighting $\widetilde w(uv)=w(uv)+2\lambda(u)+2\lambda(v)$.
  Since $R(K_6,\Z_3)=10$~\cite{HarborthPiepmeyer1994}, there is a set $X$ of six vertices in $U$ with $\widetilde w(K_X)=0$.
  Each vertex of $K_X$ occurs in five edges and $2\cdot5\equiv1\pmod3$, so
  \[
    0=\widetilde w(K_X)
    =w(K_X)+2\cdot5\sum_{z\in X}\lambda(z)
    =w(K_X)+\sum_{z\in X}\lambda(z).
  \]
  The last expression is precisely the total weight of the new clique and its edges to $A$.
  Mapping the six true twins to $X$ therefore gives a zero-sum extension $\sigma_j$.
  Iterating over $j=1,\ldots,m$ produces a zero-sum copy of $Q$ in $K_N$, which implies that $R(Q,\Z_3)\le N$.
\end{proof}

Consequently, for every $C\ge4$, a bound $R(H,\Z_3)\le |V(H)|+C$ immediately lifts to $R(Q,\Z_3)\le |V(Q)|+C$.
Apply false-three and true-six reductions to $G$ until a $3/6$-twin-reduced graph $H$ is reached.
Since every reduction is a vertex deletion, $H$ remains induced $(P_4,2K_2)$-free.
Moreover, $e(H)\equiv e(G)\equiv0\pmod3$ by \cref{prop:twin-lifting}.

The next definition isolates the local configuration that supplies two independent swaps in $H$.
Its three conditions allow the contribution of each swap to be adjusted by a separate rectangle while preventing the two swaps from interfering.

\begin{definition}\label{def:SVA}
  An ordered $6$-tuple of distinct vertices
  \[
    (x_1,y_1,z_1,x_2,y_2,z_2)
  \]
  is said to satisfy the \emph{six-vertex adjacency condition}, abbreviated \textup{(SVA)}, if the following hold for $i\in[2]$:
  \begin{enumerate}[label=\textup{(\roman*)}]
    \item $z_i$ is adjacent to exactly one of $x_i$ and $y_i$ (say $z_i$ distinguishes $x_i$ and $y_i$);
    \item $z_i$ is adjacent to both or neither of $x_{3-i}$ and $y_{3-i}$;
    \item the pairs $\{x_1,y_1\}$ and $\{x_2,y_2\}$ are complete or anticomplete to each other.
  \end{enumerate}
\end{definition}

This definition gives the main dichotomy: an \textup{(SVA)} configuration settles the Ramsey bound directly, while an \textup{(SVA)}-free graph is rigid enough to admit the structural reduction below.

\medskip
\noindent\textbf{Case 1. $H$ contains six distinct vertices satisfying condition~\textup{(SVA)}.}
\medskip

Fix an \textup{(SVA)} configuration $x_1,y_1,z_1,x_2,y_2,z_2$ and let $h=|V(H)|$.
We show that $R(H,\Z_3)\le h+4$.
Suppose to the contrary that a weighting $w:E(K_{h+4})\to\Z_3$ contains no zero-sum copy of $H$.

We first show that $w$ contains no two vertex-disjoint nonzero rectangles.
The argument follows the same switching strategy as the proof of \cref{clm:no-disjoint-nonzero-rectangles}, with the \textup{(SVA)} configuration providing the two independent switches.
Indeed, suppose that $(a_i,b_i;u_i,v_i)$, for $i\in\{1,2\}$, are two such rectangles.
Let $S=\{x_1,y_1,z_1,x_2,y_2,z_2\}$ and choose an injection
\[
  \psi:V(H)\setminus S\to
  V(K_{h+4})\setminus
  \{a_1,b_1,u_1,v_1,a_2,b_2,u_2,v_2\}.
\]
The target has $h-4$ vertices and the domain has $h-6$ vertices.

Fix $i\in\{1,2\}$, set $j=3-i$, and let $t\in\{u_i,v_i\}$.
Map $x_i,y_i,z_i$ to $a_i,b_i,t$, respectively, map the vertices outside $S$ according to $\psi$, and complete the map on the other triple using vertices of the other rectangle.
Thus, in the resulting embedding $\sigma$, we have $\sigma(x_i)=a_i$, $\sigma(y_i)=b_i$, and $\sigma(z_i)=t$.
By the \emph{$i$-th swap}, we mean replacing $\sigma$ by the embedding $\sigma^{(i)}$ defined by $\sigma^{(i)}(x_i)=b_i$, $\sigma^{(i)}(y_i)=a_i$, and $\sigma^{(i)}(v)=\sigma(v)$ for $v\notin\{x_i,y_i\}$.
Thus, the $i$-th swap exchanges only the images of $x_i$ and $y_i$ and leaves every other image unchanged.

We now compute the resulting change in copy weight.
Separate the vertices outside $S$ that distinguish $x_i$ from $y_i$ into $P_i=\bigl(N_H(x_i)\setminus N_H(y_i)\bigr)\setminus S$ and $M_i=\bigl(N_H(y_i)\setminus N_H(x_i)\bigr)\setminus S$.
A vertex $v\in P_i$ changes its incident edge weight from $w(a_i\psi(v))$ to $w(b_i\psi(v))$, whereas a vertex $v\in M_i$ produces the opposite change.
Vertices outside $S$ that are adjacent to both or neither of $x_i,y_i$ produce no change.
Hence, the total contribution from the vertices outside $S$ is the fixed quantity
\[
  c_i=
  \sum_{v\in P_i}\bigl(w(b_i\psi(v))-w(a_i\psi(v))\bigr)
  +\sum_{v\in M_i}\bigl(w(a_i\psi(v))-w(b_i\psi(v))\bigr).
\]

Condition~\textup{(SVA)(ii)} implies that the edges from $z_j$ to $\{x_i,y_i\}$ contribute no change under this swap, while condition~\textup{(SVA)(iii)} gives the same conclusion for the edges joining the two distinguished pairs.
The edge $x_i y_i$, if present, is also unchanged.
Define $\eta_i=1$ if $z_i$ is adjacent to $x_i$, and $\eta_i=-1$ if $z_i$ is adjacent to $y_i$.
Therefore, the change in copy weight under the $i$-th swap depends on $t$ only through the edge incident with $z_i$ and is
\[
  \Delta_i(t)=c_i+\eta_i\bigl(w(b_i t)-w(a_i t)\bigr),
\]
and consequently
\[
  \Delta_i(u_i)-\Delta_i(v_i)
  =-\eta_i\Delta_w(a_i,b_i;u_i,v_i)\ne0.
\]
Thus, $\Delta_i(u_i)$ and $\Delta_i(v_i)$ are distinct, so at least one of them is nonzero.
Choose $t_i\in\{u_i,v_i\}$ such that $\Delta_i(t_i)\ne0$, and set $\delta_i=\Delta_i(t_i)$.
Using these choices for both $i=1,2$, define the embedding $\sigma_0:V(H)\to V(K_{h+4})$ by $\sigma_0(x_i)=a_i$, $\sigma_0(y_i)=b_i$, and $\sigma_0(z_i)=t_i$ for $i\in\{1,2\}$, and by $\sigma_0(v)=\psi(v)$ for $v\in V(H)\setminus S$.
The disjointness of the rectangles and the choice of $\psi$ ensure that $\sigma_0$ is injective.
For $\varepsilon_1,\varepsilon_2\in\{0,1\}$, let $\sigma_{\varepsilon_1,\varepsilon_2}$ be the embedding obtained from $\sigma_0$ by performing the $i$-th swap precisely when $\varepsilon_i=1$.
Since the effects of the two swaps are independent and additive, if $z_0$ is the weight of the copy defined by $\sigma_0$, then the copy defined by $\sigma_{\varepsilon_1,\varepsilon_2}$ has weight $z_0+\varepsilon_1\delta_1+\varepsilon_2\delta_2$.
The four possible increments are therefore $0$, $\delta_1$, $\delta_2$, and $\delta_1+\delta_2$.
Again, they contain every element of $\Z_3$.
Hence, one increment equals $-z_0$, and the corresponding embedding gives a zero-sum copy of $H$, a contradiction.
Thus, $w$ contains no two vertex-disjoint nonzero rectangles.
\medskip

This exclusion yields a large affine set.
If every rectangle in $K_{h+4}$ is zero, then \cref{lem:rectangle-form} makes the whole weighting affine.
Otherwise, fix a nonzero rectangle and let $U$ be the set of the remaining $h$ vertices.
Every rectangle on $U$ is disjoint from the fixed one and is therefore zero, since $w$ contains no two vertex-disjoint nonzero rectangles.
Thus, \cref{lem:rectangle-form} makes $K_U$ affine, and \cref{lem:compression} extends it to an affine set of size at least $h+2$.
In either case, there is a set $C$ of $h+2$ host vertices on which $w(uv)=\phi_u+\phi_v+c$ for distinct $u,v\in C$ and suitable labels $\phi_u,c\in\Z_3$.

By \cref{lem:degree-completion}, there is an embedding $\sigma:V(H)\to C$ satisfying $\sum_{v\in V(H)}\deg_H(v)\cdot\phi_{\sigma(v)}=0$.
For the copy $H_\sigma$ determined by $\sigma$, we have
\[
  \begin{aligned}
    w(H_\sigma)
     & =\sum_{xy\in E(H)}w\bigl(\sigma(x)\sigma(y)\bigr)              \\
     & =\sum_{v\in V(H)}\deg_H(v)\cdot\phi_{\sigma(v)}+c\cdot e(H)=0.
  \end{aligned}
\]
This contradicts the choice of $w$.
Thus, $R(H,\Z_3)\le h+4$, and \cref{prop:twin-lifting} yields $R(G,\Z_3)\le |V(G)|+4$.

\medskip
\noindent\textbf{Case 2. $H$ contains no six distinct vertices satisfying condition~\textup{(SVA)}.}
\medskip

In this case, the preceding \textup{(SVA)}-based switching argument is unavailable.
We collect the four properties that remain after all reductions into the following notion.

A graph $Q$ is \emph{irreducible} if it is induced $(P_4,2K_2)$-free, $3/6$-twin-reduced, \textup{(SVA)}-free, and satisfies $3\mid e(Q)$.
Thus, $H$ is irreducible.
If $H=K_0$, then \cref{prop:twin-lifting} and the convention $R(K_0,\Z_3)=0$ give $R(G,\Z_3)\le |V(G)|+4$.
Hence, we may assume that $H$ is nonempty.
We derive the structural forms needed for all such graphs in \cref{sec:main2-characterization} and establish the Ramsey bounds needed to finish this case in \cref{sec:main2-irreducible-bounds}, with the two exceptional proofs given in Appendices~\ref{app:k5minus-bound} and~\ref{app:k222-bound}.

\subsection{Structural reduction of irreducible graphs}
\label{sec:main2-characterization}

Our aim in this subsection is to extract from the rooted-forest representation only the structural forms needed for the zero-sum Ramsey bounds.
The following representation theorem provides the rooted-forest model for $\overline H$.
It turns nonedges of $H$ into ancestor relations in a rooted forest.

\begin{theorem}[Golumbic \cite{Golumbic1978}]
  \label{thm:trivially-perfect}
  A graph $Q$ is induced $\{P_4,C_4\}$-free if and only if there is a rooted forest $F$ on $V(Q)$ such that two distinct vertices are adjacent in $Q$ exactly when one is an ancestor of the other in $F$.
\end{theorem}

Two distinct vertices of a rooted forest are called \emph{comparable} if one is an ancestor of the other, and \emph{incomparable} otherwise.
Because $P_4$ is self-complementary and $\overline{2K_2}=C_4$, the complement $\overline H$ is induced $\{P_4,C_4\}$-free.
Hence, \cref{thm:trivially-perfect} gives a rooted forest $F$ on $V(H)$ in which two distinct vertices are adjacent in $H$ precisely when they are incomparable in $F$.

All rooted-forest terminology refers to the orientation away from the roots.
A vertex of a rooted forest is called \emph{branching} if it has at least two children, \emph{unary} if it has exactly one child, and a \emph{leaf} if it has no children.
A \emph{root--leaf path} is a path from the root of a component to one of its leaves.
The following lemma is the structural core of the reduction: it translates the twin restrictions and the absence of an \textup{(SVA)} configuration into four constraints that force $F$ to have at most two branching vertices.

\begin{lemma}
  \label{lem:forest-restrictions}
  Let $H$ be an irreducible graph, and let $F$ be a rooted forest on $V(H)$ such that $uv\in E(H)$ if and only if $u$ and $v$ are incomparable in $F$.
  Then the following properties hold.
  \begin{enumerate}[label=\textup{(\roman*)}]
    \item Every component of $F$ without a branching vertex has at most two vertices.
    \item Every vertex of $F$ has at most five leaf children.
    \item No root--leaf path contains two consecutive unary vertices.
    \item The forest $F$ has at most two branching vertices.
          If it has two, then one is the parent of the other.
  \end{enumerate}
\end{lemma}

\begin{proof}
  A component with no branching vertex is a rooted path, so its vertices are pairwise comparable and hence form an independent set in $H$.
  If such a component $C$ has at least three vertices, then any three vertices of $C$ have the same open neighborhood $V(H)\setminus V(C)$ and hence are false twins, contradicting the irreducibility of $H$.
  This proves \textup{(i)}.
  Similarly, six leaf children of one vertex are pairwise incomparable and have identical closed neighborhoods in $H$; hence, they are true twins, again contradicting the irreducibility of $H$.
  This proves \textup{(ii)}.
  Next, suppose that $a,b$ are consecutive unary vertices.
  Since $b$ is unary, it has a unique child $c$, and $a,b,c$ occur consecutively on a root--leaf path.
  Every other vertex is then either comparable with all three or incomparable with all three.
  Thus, $N_H(a)=N_H(b)=N_H(c)$, so $\{a,b,c\}$ is a set of false twins.
  This contradicts the irreducibility of $H$ and proves \textup{(iii)}.

  It remains to prove \textup{(iv)}.
  We first show that all branching vertices are comparable.
  Suppose that branching vertices $p_1,p_2$ are incomparable.
  For $i\in[2]$, choose distinct children $a_i,b_i$ of $p_i$.
  The tuple
  \[
    (p_1,a_1,b_1,p_2,a_2,b_2)
  \]
  satisfies \textup{(SVA)}: within each triple, $b_i$ distinguishes $p_i$ from $a_i$, while every vertex in one triple is adjacent to every vertex in the other, since the vertices in different triples are incomparable in $F$.
  This contradicts the irreducibility of $H$, so all branching vertices are comparable and therefore lie on one root--leaf path.

  Now let $p,q$ be branching vertices with $p$ an ancestor of $q$.
  If $p$ is not the parent of $q$, let $c$ be the child of $p$ toward $q$, choose another child $d$ of $p$, and choose distinct children $a,b$ of $q$.
  Then
  \[
    (p,c,d,q,a,b)
  \]
  satisfies \textup{(SVA)}: $d$ distinguishes $p,c$, and $b$ distinguishes $q,a$.
  Moreover, every vertex of $\{p,c\}$ is comparable with every vertex of $\{q,a\}$, whereas $d$ is incomparable with both $q,a$ and $b$ is comparable with both $p,c$.
  Thus, the two distinguished pairs are anticomplete, $d$ is adjacent to both vertices of the second pair, and $b$ is adjacent to neither vertex of the first pair, verifying \textup{(SVA)(ii)--(iii)}.
  This contradicts the irreducibility of $H$ and shows that, for any two branching vertices, one must be the parent of the other.
  In particular, three branching vertices cannot lie on the same root--leaf path, since the first cannot be the parent of the third.
  Hence, $F$ has at most two branching vertices.
\end{proof}

Now, we charcterize the structure of nonempty irreducible graphs.

\begin{theorem}
  \label{thm:irreducible-classification}
  Let $H$ be a nonempty irreducible graph.
  Then one of the following holds.
  \begin{enumerate}[label=\textup{(\roman*)}]
    \item $H\in\{K_1,2K_1\}$.
    \item There are $q_0\in\{0,1,2\}$ and $J\in\{K_3,K_4,K_5-e,K_{2,2,2}\}$ such that $H=J\cup q_0K_1$.
    \item There are $q_0\in\{0,1,2\}$, a nonempty graph $A$, an integer
          $t\in\{1,2\}$, and $F_0\in\{P_3,C_4\}\cup\{K_s:s\ge2\}$ such that $H=\bigl(A+(tK_1\cup F_0)\bigr)\cup q_0K_1$.
  \end{enumerate}
\end{theorem}

\begin{proof}
  Choose a rooted forest $F$ representing $H$.
  By \cref{lem:forest-restrictions} (iv), the forest $F$ has $0$, $1$, or $2$ branching vertices.

  \medskip
  \noindent\textbf{Case 1. $F$ has no branching vertex.}
  \medskip

  By \cref{lem:forest-restrictions} (i), every component of $F$ has at most two vertices.
  Let $a$ and $b$ be the numbers of two-vertex and singleton components of $F$, respectively.
  The singleton components are pairwise true twins in $H$, so $b\le5$.
  Writing the $i$th two-vertex component as $q_iq_i'$ and the $j$th singleton component as $u_j$, the only nonedges of $H$ are the pairs $q_iq_i'$.
  The tuples
  \[
    (q_1,q_2,q_1',q_3,q_4,q_3');\qquad
    (q_1,q_2,q_1',q_3,u_1,q_3');\qquad
    (q_1,u_1,q_1',q_2,u_2,q_2')
  \]
  satisfy \textup{(SVA)} when, respectively, $a\ge4$, $a=3$ and $b\ge1$, or $a=2$ and $b\ge2$.
  Since $H$ is \textup{(SVA)}-free, we have
  \[
    a\le1;\qquad a=2\text{ and }b\le1;\qquad\text{or}\qquad a=3\text{ and }b=0.
  \]
  Moreover,
  \[
    e(H)=\binom{2a+b}{2}-a.
  \]
  If $a=0$, the conditions $1\le b\le5$ and $3\mid\binom b2$ give $b\in\{1,3,4\}$.
  If $a=1$, the conditions $0\le b\le5$ and
  $3\mid\bigl(\binom{b+2}{2}-1\bigr)$ give $b\in\{0,3\}$.
  Neither value $b\in\{0,1\}$ works when $a=2$, while $(a,b)=(3,0)$ gives $e(H)=12$.
  Thus,
  \[
    H\in\{K_1,2K_1,K_3,K_4,K_5-e,K_{2,2,2}\},
  \]
  as required.

  \medskip
  \noindent\textbf{Case 2. $F$ has one branching vertex.}
  \medskip

  Let $p$ be the unique branching vertex.
  By \cref{lem:forest-restrictions} (i) and (iii), the vertex $p$ has at most one ancestor, every arm below $p$ has length at most two, and every other component has order at most two.
  Let $a\in\{0,1\}$ indicate whether $p$ has an ancestor; let $r$ and $s$ count the arms of lengths one and two; and let $b$ and $c$ count the external components of orders one and two.
  Then $r+s\ge2$ and $r\le5$.
  Denote the length-one arms by $u_i$, the length-two arms by $pv_jw_j$, the singleton external components by $e_i$, and the two-vertex external components by $d_jd_j'$.

  Let $P$ consist of $p$ and its optional ancestor, let $L$ be the set of vertices in the arms, and let $X$ be the union of the external components.
  Then $P$ is anticomplete to $L$, both $P$ and $L$ are complete to $X$, and
  \[
    H[L]=K_{r+2s}-sK_2.
  \]
  The three tuples
  \[
    (v_1,v_2,w_1,v_3,v_4,w_3);\qquad
    (v_1,u_1,w_1,v_2,v_3,w_2);\qquad
    (v_1,u_1,w_1,v_2,u_2,w_2)
  \]
  satisfy \textup{(SVA)} when, respectively, $s\ge4$, $r\ge1$ and $s\ge3$, or $r\ge2$ and $s\ge2$.
  Hence none of these three parameter conditions can occur.

  Suppose first that $X\ne\emptyset$, and fix $e\in X$.
  The tuples
  \[
    (e,u_1,p,u_2,v_1,w_1);\qquad
    (e,u_1,p,v_1,v_2,w_1);\qquad
    (e,v_1,p,v_2,v_3,w_2)
  \]
  satisfy \textup{(SVA)} when, respectively, $r\ge2$ and $s\ge1$, $r\ge1$ and $s\ge2$, or $s\ge3$.
  Since $r+s\ge2$, it follows that
  \[
    H[L]\in\{K_r:r\ge2\}\cup\{P_3,C_4\}.
  \]
  Indeed, the three possibilities are $s=0$, $(r,s)=(1,1)$, and $(r,s)=(0,2)$.
  Taking $A=H[X]$, $t=|P|=1+a$, $F_0=H[L]$, and $q_0=0$ gives conclusion~\textup{(iii)}.

  Now suppose that $X=\emptyset$.
  The restrictions above leave
  \[
    \begin{array}{c|c}
      s & \text{possible values of }r \\ \hline
      0 & 2\le r\le5                  \\
      1 & 1\le r\le5                  \\
      2 & r\in\{0,1\}                 \\
      3 & r=0.
    \end{array}
  \]
  Since $P$ is independent and anticomplete to $L$,
  \[
    e(H)=e(H[L])=\binom{r+2s}{2}-s.
  \]
  For $s=0$, divisibility by three gives $r\in\{3,4\}$; for $s=1$, it gives $r=3$; for $s=2$, neither $r=0$ nor $r=1$ works; and for $(r,s)=(0,3)$ the edge number is $12$.
  Consequently,
  \[
    H[L]\in\{K_3,K_4,K_5-e,K_{2,2,2}\}.
  \]
  Since $P=(1+a)K_1$, conclusion~\textup{(ii)} follows with $q_0=1+a\in\{1,2\}$.

  \medskip
  \noindent\textbf{Case 3. $F$ has two branching vertices.}
  \medskip

  By \cref{lem:forest-restrictions} (iv), the two branching vertices, $p$ and $q$, are adjacent in $F$; let $p$ be the parent of $q$.
  As in the preceding case, $p$ has at most one ancestor, every side arm of $p$ and every arm below $q$ has length at most two, and every other component has order at most two.
  Let $a\in\{0,1\}$ indicate whether $p$ has an ancestor; let $r,s$ count the side arms of $p$ of lengths one and two; let $k,\ell$ count the arms below $q$ of lengths one and two; and let $b,c$ count the external components of orders one and two.
  Then $r+s\ge1$, $k+\ell\ge2$, and $2\le k\le5$ whenever $\ell=0$.

  Choose children $v_1,v_2$ of $q$ in distinct arms and a side-arm child $u$ of $p$.
  The tuples
  \[
    (q,v_1,v_2,u,e,p);\qquad
    (p,q,u,v_1,v_2,w_1);\qquad
    (q,v_1,v_2,u_1,u_2,x_1)
  \]
  satisfy \textup{(SVA)} when, respectively, an external vertex $e$ is present, $qv_1w_1$ is a length-two arm below $q$, or $pu_1x_1$ is a length-two side arm of $p$ and $u_2$ is a child in another side arm.
  Therefore $b=c=0$, $\ell=0$, and either $s=0$ or $(r,s)=(0,1)$.
  In the latter case,
  \[
    e(H)=\binom{k}{2}+2(k+1)\in\{7,11,16,22\}
    \qquad(2\le k\le5),
  \]
  contradicting $3\mid e(H)$.
  Hence $s=\ell=0$.

  Let $I$ consist of $p$ and its optional ancestor, let $U$ be the set of the $r$ side leaves of $p$, and let $V$ be the set of the $k$ leaves below $q$.
  Then
  \[
    H=(1+a)K_1\cup\bigl(K_r+(K_1\cup K_k)\bigr).
  \]
  Taking $A=K_r$, $t=1$, $F_0=K_k$, and $q_0=1+a\in\{1,2\}$ gives conclusion~\textup{(iii)}.
\end{proof}

\subsection{zero-sum Ramsey bounds for the irreducible graphs}
\label{sec:main2-irreducible-bounds}

The following lemma reduces the finite verification to graphs without isolated vertices by proving that adjoining each isolated vertex increases the required host order by at most one.

\begin{lemma}
  \label{lem:add-isolates}
  For every graph $J$ with $3\mid e(J)$ and every integer $t\ge0$,
  \[
    R(J\cup tK_1,\Z_3)
    \le
    \max\bigl\{R(J,\Z_3),\,|V(J)|+t\bigr\}.
  \]
\end{lemma}

\begin{proof}
  Set $N=\max\bigl\{R(J,\Z_3),\,|V(J)|+t\bigr\}$ and consider an arbitrary weighting of $K_N$.
  Since $N\ge R(J,\Z_3)$, the weighting contains a zero-sum copy of $J$.
  This copy uses $|V(J)|$ host vertices, while $N-|V(J)|\ge t$.
  Mapping the $t$ isolated vertices to any $t$ unused host vertices gives a zero-sum copy of $J\cup tK_1$.
\end{proof}

We now apply the structural description from \cref{thm:irreducible-classification}.
If $H\in\{K_1,2K_1\}$, then $R(H,\Z_3)=|V(H)|$, and \cref{prop:twin-lifting} gives $R(G,\Z_3)\le |V(G)|+4$.
Next, let $H=K_3\cup qK_1$, where $q\in\{0,1,2\}$.
By \cref{lem:add-isolates},
\[
  R(H,\Z_3)\le\max\{11,3+q\}=11.
\]
Since $|V(G)|\ge6$, \cref{prop:twin-lifting} gives
\[
  R(G,\Z_3)\le
  \max\{11,|V(G)|+4\}\le |V(G)|+5.
\]
Similarly, if $H=K_4\cup qK_1$ for some $q\in\{0,1,2\}$, then
\[
  R(H,\Z_3)\le\max\{7,4+q\}=7,
\]
and \cref{prop:twin-lifting} gives $R(G,\Z_3)\le |V(G)|+4$.

It remains to treat $K_5-e$, $K_{2,2,2}$, and the graphs of the form $A+(tK_1\cup F)$ appearing in conclusion~\textup{(iii)} of \cref{thm:irreducible-classification}; adjoining the isolated vertices is then handled by \cref{lem:add-isolates}.
The following lemma treats the last family uniformly.

\begin{lemma}
  \label{lem:two-level}
  Let $A$ be a nonempty graph, let $t\in\{1,2\}$, and let $F\in\{P_3,C_4\}\cup\{K_s:s\ge2\}$.
  If $G=A+(tK_1\cup F)$ and $3\mid e(G)$, then $R(G,\Z_3)\le |V(G)|+5$.
\end{lemma}

\begin{proof}
  Let $n=|V(G)|$, let $m=|V(A)|$, and let $f=|V(F)|$.
  Let $K=K_{n+5}$.
  Suppose, for a contradiction, that a weighting $w:E(K)\to\Z_3$ contains no zero-sum copy of $G$.
  Since $m\ge1$, $t\ge1$, and $f\ge2$, we have $n\ge4$.
  If $n=4$, then $A=K_1$, $t=1$, and $F=K_2$, so $e(G)=4$, contrary to $3\mid e(G)$.
  Hence, $n\ge5$.

  For every $m$-set $P\subseteq V(K)$, fix a copy $A_P$ of $A$ with vertex set $P$, and let $\alpha_P=w(A_P)$.
  For each $z\in V(K)\setminus P$, define
  \[
    \lambda_P(z)=\sum_{p\in P}w(pz).
  \]
  Consider an arbitrary partition $V(K)\setminus P=U\mathbin{\dot\cup}W$ with $|U|=t+2$ and $|W|=f+3$.
  Define
  \[
    X_P(U)=\left\{\sum_{u\in I}\lambda_P(u):I\subseteq U,\ |I|=t\right\}
  \]
  and
  \[
    Y_{P,F}(W)=\left\{w(F_\sigma)+\sum_{x\in V(F)}\lambda_P(\sigma(x)):\sigma:V(F)\hookrightarrow W\right\},
  \]
  where $F_\sigma$ is the copy of $F$ determined by $\sigma$.
  The image of $tK_1$ can be chosen in $U$ independently of the image of $F$ in $W$ because no edge of $G$ joins $tK_1$ and $F$.
  Consequently, every element of $\alpha_P+X_P(U)+Y_{P,F}(W)$ is the weight of a copy of $G$, so this sumset does not contain zero.

  A direct check of the three two-element subsets of $\Z_3$ shows that the sum of any two subsets of $\Z_3$ of size at least two is $\Z_3$.
  It follows that
  \begin{equation}
    |X_P(U)|=1\text{ or }|Y_{P,F}(W)|=1.
    \label{eq:two-level-spectrum-dichotomy}
  \end{equation}
  Moreover, if $X_P(U)$ is a singleton, then $\lambda_P$ is constant on $U$.
  For $t=1$, this follows directly from the definition of $X_P(U)$.
  For $t=2$, equality of all two-term sums gives $\lambda_P(v)=\lambda_P(z)$ by comparing $\lambda_P(u)+\lambda_P(v)$ and $\lambda_P(u)+\lambda_P(z)$ for three distinct vertices $u,v,z\in U$.

  The next claim determines the consequence of the second possibility in \eqref{eq:two-level-spectrum-dichotomy}.

  \begin{claim}
    \label{clm:two-level-module-rigidity}
    If $P\subseteq V(K)$ has size $m$, $W\subseteq V(K)\setminus P$ has size $f+3$, and $Y_{P,F}(W)$ is a singleton, then the restriction of $w$ to $K_W$ is affine.
  \end{claim}

  \begin{proof}
    Write $\lambda(z)=\lambda_P(z)$ for every $z\in W$.
    We prove that $w$ is affine on $W$ by considering the three possible forms of $F$.

    \medskip
    \noindent\textbf{Case 1. $F=K_s$.}
    \medskip

    The copies of $F=K_s$ in $W$ are precisely the graphs $K_S$ over the $s$-sets $S\subseteq W$.
    Since $Y_{P,F}(W)$ is a singleton, there exists $\gamma\in\Z_3$ such that $w(K_S)+\sum_{z\in S}\lambda(z)=\gamma$ for every $s$-set $S\subseteq W$.
    Fix distinct vertices $x,y\in W$, and let $T\subseteq W\setminus\{x,y\}$ be an arbitrary $(s-1)$-set.
    Applying the identity defining $\gamma$ to $S=T\cup\{x\}$ and $S=T\cup\{y\}$ gives
    \begin{equation}
      \lambda(x)-\lambda(y)+\sum_{u\in T}\bigl(w(xu)-w(yu)\bigr)=0.
      \label{eq:two-level-clique-difference}
    \end{equation}
    Let $z,z'\in W\setminus\{x,y\}$ be distinct.
    Since $|W\setminus\{x,y,z,z'\}|=s-1$, choose an $(s-2)$-set $R\subseteq W\setminus\{x,y,z,z'\}$.
    Comparing \eqref{eq:two-level-clique-difference} for $T=R\cup\{z\}$ and $T=R\cup\{z'\}$ gives $w(xz)-w(yz)=w(xz')-w(yz')$.
    Equivalently, $\Delta_w(x,y;z,z')=0$.
    Since $x,y,z,z'$ were arbitrary, $\Delta_w(a,b;u,v)=0$ for all distinct $a,b,u,v\in W$.
    As $|W|=s+3\ge5$, \cref{lem:rectangle-form} shows that $w$ is affine on $W$.

    \medskip
    \noindent\textbf{Case 2. $F=P_3$.}
    \medskip

    For distinct vertices $x,y\in W$, define $g_y(x)=w(xy)+\lambda(x)+2\lambda(y)$.
    If $x,y,z$ are distinct, then
    \[
      g_y(x)+g_y(z)=w(xy)+w(yz)+\lambda(x)+\lambda(y)+\lambda(z).
    \]
    This sum is the element of $Y_{P,F}(W)$ obtained by mapping the middle vertex of $P_3$ to $y$ and its endpoints to $x$ and $z$.
    Since $Y_{P,F}(W)$ is a singleton, all sums $g_y(x)+g_y(z)$ with distinct $x,z\in W\setminus\{y\}$ are equal.
    Given distinct $x,x'\in W\setminus\{y\}$, choose $z\in W\setminus\{x,x',y\}$.
    Comparing the sums associated with $\{x,z\}$ and $\{x',z\}$ gives $g_y(x)+g_y(z)=g_y(x')+g_y(z)$, and hence $g_y(x)=g_y(x')$.
    Hence, $g_y$ is constant on $W\setminus\{y\}$; denote its value by $c_y$.
    For distinct $x,y\in W$, we have $w(xy)=c_y-\lambda(x)-2\lambda(y)$.
    Since $w(xy)=w(yx)$, it follows that $c_y-\lambda(x)-2\lambda(y)=c_x-\lambda(y)-2\lambda(x)$, which simplifies in $\Z_3$ to $c_y-\lambda(y)=c_x-\lambda(x)$.
    Hence, $c_y-\lambda(y)$ is independent of $y$; denote its common value by $c$.
    Then $c_y=c+\lambda(y)$ for every $y\in W$, and therefore $w(xy)=c-\lambda(x)-\lambda(y)$ for all distinct $x,y\in W$.
    Taking $\phi_y=-\lambda(y)$ for every $y\in W$, we obtain $w(xy)=\phi_x+\phi_y+c$, so $w$ is affine on $W$.

    \medskip
    \noindent\textbf{Case 3. $F=C_4$.}
    \medskip

    Define a weighting $\widetilde w$ on $K_W$ by $\widetilde w(xy)=w(xy)+2\lambda(x)+2\lambda(y)$.
    Since every vertex of a $4$-cycle has degree two, every $4$-cycle $Q$ satisfies $\widetilde w(Q)=w(Q)+\sum_{z\in V(Q)}\lambda(z)$.
    Thus, all $4$-cycles have the same $\widetilde w$-weight.
    For four distinct vertices $a,b,c,d\in W$, compare the $4$-cycles $abcda$ and $abdca$.
    Since all $4$-cycles have the same $\widetilde w$-weight, these two $4$-cycles have equal weight, and hence
    \[
      0=\widetilde w(bc)+\widetilde w(ad)-\widetilde w(bd)-\widetilde w(ac)=\Delta_{\widetilde w}(b,a;c,d).
    \]
    Since $a,b,c,d$ are arbitrary, every rectangle of $\widetilde w$ has value zero.
    For any distinct $a,b,u,v\in W$, substituting the definition of $\widetilde w$ gives
    \[
      \begin{aligned}
        \Delta_{\widetilde w}(a,b;u,v)
         & =\widetilde w(au)+\widetilde w(bv)-\widetilde w(av)-\widetilde w(bu) \\
         & =\Delta_w(a,b;u,v)
        +2\bigl(\lambda(a)+\lambda(u)+\lambda(b)+\lambda(v)
        -\lambda(a)-\lambda(v)-\lambda(b)-\lambda(u)\bigr)                      \\
         & =\Delta_w(a,b;u,v).
      \end{aligned}
    \]
    Hence, every rectangle of $w$ also has value zero.
    By \cref{lem:rectangle-form}, the weighting $w$ is affine on $W$.

    \medskip
    This proves the claim.
  \end{proof}

  A nonzero rectangle contained in $W$ prevents $w$ from being affine on $W$.
  Hence, \cref{clm:two-level-module-rigidity} rules out $|Y_{P,F}(W)|=1$, and \eqref{eq:two-level-spectrum-dichotomy} forces $|X_P(U)|=1$; consequently, $\lambda_P$ is constant on $U$.
  The following claim extends the conclusion that $\lambda_P$ is constant on $U$ to all vertices outside a fixed nonzero rectangle.

  \begin{claim}
    \label{clm:two-level-star-constancy}
    Let $Q$ be a nonzero rectangle and let $P\subseteq V(K)\setminus V(Q)$ be an $m$-set.
    Then $\lambda_P$ is constant on $V(K)\setminus(P\cup V(Q))$.
  \end{claim}

  \begin{proof}
    Let $Z=V(K)\setminus(P\cup V(Q))$, so $|Z|=t+f+1\ge t+3$.
    Choose an arbitrary $(t+2)$-set $U\subseteq Z$ and let $W=V(Q)\cup(Z\setminus U)$.
    Then $|W|=f+3$, and $W$ contains the nonzero rectangle $Q$.
    An affine weighting has no nonzero rectangle, so $w$ is not affine on $W$.
    Consequently, \cref{clm:two-level-module-rigidity} implies that $Y_{P,F}(W)$ is not a singleton.
    By \eqref{eq:two-level-spectrum-dichotomy}, the set $X_P(U)$ is a singleton, so $\lambda_P$ is constant on $U$.
    Since $U$ was arbitrary and every two vertices of $Z$ lie in a common $(t+2)$-subset of $Z$, the function $\lambda_P$ is constant on $Z$.
  \end{proof}

  We next prove that $w$ contains no two vertex-disjoint nonzero rectangles.
  Suppose, for a contradiction, that $Q_0=(x,y;u,v)$ and $Q_1$ are vertex-disjoint nonzero rectangles.
  Since $|V(K)\setminus(V(Q_0)\cup V(Q_1))|=m+t+f-3\ge m-1$, choose an $(m-1)$-set $S$ outside the two rectangles.
  Let $P_x=S\cup\{x\}$ and $P_y=S\cup\{y\}$.
  The sets $P_x$ and $P_y$ are disjoint from $V(Q_1)$, so \cref{clm:two-level-star-constancy} gives $\lambda_{P_x}(u)=\lambda_{P_x}(v)$ and $\lambda_{P_y}(u)=\lambda_{P_y}(v)$.
  By the definition of $\lambda_P$, these equalities are
  \[
    \begin{aligned}
      \sum_{s\in S}w(su)+w(xu)
       & =\sum_{s\in S}w(sv)+w(xv), \\
      \sum_{s\in S}w(su)+w(yu)
       & =\sum_{s\in S}w(sv)+w(yv).
    \end{aligned}
  \]
  Their difference gives
  \[
    0=w(xu)+w(yv)-w(xv)-w(yu)=\Delta_w(x,y;u,v),
  \]
  contrary to the choice of $Q_0$.
  Hence, $w$ contains no two vertex-disjoint nonzero rectangles.

  If every rectangle in $K$ has value zero, then \cref{lem:rectangle-form} shows that $w$ is affine on $V(K)$.
  Otherwise, fix a nonzero rectangle $Q$ and let $U_0=V(K)\setminus V(Q)$.
  Every rectangle contained in $U_0$ has value zero because a nonzero rectangle in $U_0$ would be vertex-disjoint from $Q$.
  Moreover, $|U_0|=n+1\ge6$, so \cref{lem:rectangle-form} shows that $w$ is affine on $U_0$.
  The partition $V(K)=U_0\cup V(Q)$ satisfies the hypotheses of \cref{lem:compression}, and hence there is a set $D\subseteq V(K)$ with $|D|\le2$ such that $w$ is affine on $V(K)\setminus D$.
  In either case, there is a set $C\subseteq V(K)$ with $|C|\ge n+3$ on which $w$ is affine.

  Choose an $(n+2)$-set $C_0\subseteq C$, and write $w(xy)=\phi_x+\phi_y+c$ for all distinct $x,y\in C_0$.
  Since $3\mid e(G)$, \cref{lem:degree-completion} applied to the labels $(\phi_x)_{x\in C_0}$ gives an injection $\sigma:V(G)\hookrightarrow C_0$ such that
  \[
    \sum_{v\in V(G)}\deg_G(v)\phi_{\sigma(v)}=0.
  \]
  For the copy $G_\sigma$ determined by $\sigma$, we obtain
  \[
    w(G_\sigma)=\sum_{xy\in E(G)}w\bigl(\sigma(x)\sigma(y)\bigr)=\sum_{v\in V(G)}\deg_G(v)\phi_{\sigma(v)}+e(G)c=0.
  \]
  This contradicts the assumption that $w$ contains no zero-sum copy of $G$ and proves the lemma.
\end{proof}

The two exceptional graphs require separate arguments, given in Appendices~\ref{app:k5minus-bound} and~\ref{app:k222-bound}.

\begin{lemma}
  \label{lem:k5minus}
  We have $R(K_5-e,\Z_3)\le 10=|V(K_5-e)|+5$.
\end{lemma}

\begin{lemma}
  \label{lem:k222}
  We have $R(K_{2,2,2},\Z_3)\le11=|V(K_{2,2,2})|+5$.
\end{lemma}

We now finish the second case of \cref{thm:main}.
The graphs $K_1$, $2K_1$, and the graphs $K_3\cup qK_1$ and $K_4\cup qK_1$ were handled above.
For $J\in\{K_5-e,K_{2,2,2}\}$ and $q\in\{0,1,2\}$, \cref{lem:add-isolates,lem:k5minus,lem:k222} gives
\[
  R(J\cup qK_1,\Z_3)\le |V(J\cup qK_1)|+5.
\]
Likewise, if $J=A+(tK_1\cup F)$ with $F\in\{P_3,C_4\}\cup\{K_s:s\ge2\}$, then \cref{lem:two-level} gives $R(J,\Z_3)\le |V(J)|+5$, and \cref{lem:add-isolates} gives
\[
  R(J\cup qK_1,\Z_3)\le |V(J\cup qK_1)|+5
\]
for $q=0,1,2$.
Thus, \cref{thm:irreducible-classification} covers every remaining possibility for $H$, and \cref{prop:twin-lifting} yields
\[
  R(G,\Z_3)\le |V(G)|+5.
\]
This completes the proof of \cref{thm:main}.

\section{Proof of \cref{cor:main}}
\label{sec:proof-cor-main}

Let $G$ be a graph with $3\mid e(G)$, and set $n=|V(G)|$.
In this section, we prove that $R(G,\Z_3)\le n+8$ with equality if and only if $G$ is $K_3$.

First suppose that $n\ge6$.
By \cref{thm:main}, $R(G,\Z_3)\le n+5<n+8$.
If $n\le2$, then $3\mid e(G)$ implies $e(G)=0$.
Hence, $G=nK_1$ and $R(G,\Z_3)=n<n+8$.

It remains to consider $3\le n\le5$.
To bring $G$ into the range of \cref{thm:main} without adding any edges, let $\widehat G=G\cup(6-n)K_1$.
Then $\widehat G$ has six vertices and $e(\widehat G)=e(G)\equiv0\pmod3$, so \cref{thm:main} gives $R(\widehat G,\Z_3)\le11$.
Since the added vertices are isolated, deleting the corresponding vertices from any zero-sum copy of $\widehat G$ leaves a zero-sum copy of $G$ with the same edge set.
Hence,
\[
  R(G,\Z_3)\le R(\widehat G,\Z_3)\le11.
\]

If $n\in\{4,5\}$, then $R(G,\Z_3)\le11<n+8$, so the inequality is again strict.
Finally, suppose that $n=3$.
Since $3\mid e(G)$ and $0\le e(G)\le3$, either $e(G)=0$ or $e(G)=3$.
Thus, $G=3K_1$ or $G=K_3$.
In the first case, $R(3K_1,\Z_3)=3<11=|V(3K_1)|+8$.
In the second case, the known value $R(K_3,\Z_3)=11$~\cite{ChungGraham1983} shows that $R(K_3,\Z_3)=|V(K_3)|+8$, so equality holds.

Combining all cases, we conclude that $R(G,\Z_3)\le |V(G)|+8$ with equality if and only if $G\cong K_3$.
This completes the proof of \cref{cor:main}.

\paragraph{Acknowledgments.}

We are grateful to Professor Yair Caro for bringing the problem studied in this paper to our attention, directing us to the relevant literature, and offering valuable suggestions that improved the paper.

\paragraph{Declaration on the use of AI.}

Generative AI was used only to assist with the structural reduction in \cref{thm:irreducible-classification} and with language polishing.
The authors independently verified and revised all AI-assisted material and take full responsibility for every mathematical claim and for the final manuscript.

\appendix

\section{The bound for $K_5-e$}
\label{app:k5minus-bound}

\begin{proof}[Proof of \cref{lem:k5minus}]
  Suppose for contradiction that $w:E(K_{10})\to\Z_3$ contains no zero-sum copy of $K_5-e$.
  For a five-set $S$, write
  \[
    T(S)=\sum_{e\in E(K_S)}w(e).
  \]
  Every copy of $K_5-e$ on the vertex set $S$ is obtained by omitting one edge $e$ of $K_S$, and its weight is $T(S)-w(e)$.
  Hence
  \begin{equation}
    T(S)\ne w(e)
    \text{ for every }e\in E(K_S).
    \label{eq:k5-noedge}
  \end{equation}
  Since $T(S)\in\Z_3$, if all three residues occurred among the edge weights of $K_S$, then some edge $e\in E(K_S)$ would satisfy $w(e)=T(S)$, contradicting \eqref{eq:k5-noedge}.
  Hence the edge weights of $K_S$ take at most two values.

  We now show that this restriction on every five-set forces the edge weights of $K_{10}$ to take at most two values.
  Suppose otherwise, and choose edges $e_0,e_1,e_2$ of weights $0,1,2$, respectively.
  Their endpoint union has more than five vertices; otherwise, it is contained in a five-set $S$, and $K_S$ has edges of all three weights, a contradiction.
  Hence the three edges are pairwise disjoint.
  Write $e_i=a_i a_i'$ for $i=0,1,2$.
  Consider the cross edge $a_0a_1$.
  If its weight is $0$, then $a_0a_1,e_1,e_2$ use all three residues on five vertices.
  If its weight is $1$, then $e_0,a_0a_1,e_2$ do so.
  If its weight is $2$, then $e_0,e_1,a_0a_1$ do so on four vertices.
  Every case contradicts the fact that the edge weights on each five-set take at most two values.
  Thus at most two residues occur globally.

  The weighting $w$ is not constant.
  Indeed, if $w(e)=d\in\Z_3$ for every $e\in E(K_{10})$, then every copy $H$ of $K_5-e$ satisfies $w(H)=9d=0$, contrary to our assumption.
  Therefore, the weighting $w$ takes exactly two values.
  We may assume that every edge has weight $0$ or $1$.
  Indeed, an affine permutation $x\mapsto ux+b$ of $\Z_3$, with $u\ne0$, sends the two values of $w$ to $0$ and $1$.
  Multiplication by $u$ preserves zero sums, while adding $b$ to every edge changes the weight of each copy of $K_5-e$ by $9b=0$.

  We now determine the total weight on each five-set.
  Let $S$ be a five-set.
  Both $0$ and $1$ occur among the weights of $E(K_S)$; otherwise, every edge of $K_S$ would have the same weight $d\in\{0,1\}$, and each copy of $K_5-e$ on $S$ would have weight $9d=0$, contrary to our assumption.
  Since both values occur on $E(K_S)$, \eqref{eq:k5-noedge} implies $T(S)\notin\{0,1\}$.
  Consequently,
  \begin{equation}
    \sum_{e\in E(K_S)}w(e)=2
    \label{eq:k5-five-sum}
  \end{equation}
  for every five-set $S$.

  Fix distinct vertices $p,q$ and a four-set $A\subseteq V(K_{10})\setminus\{p,q\}$.
  Subtracting \eqref{eq:k5-five-sum} for $A\cup\{p\}$ and $A\cup\{q\}$ gives
  \begin{equation}
    \sum_{t\in A}\bigl(w(pt)-w(qt)\bigr)=0.
    \label{eq:k5-row-sum}
  \end{equation}
  For each $t\in V(K_{10})\setminus\{p,q\}$, write $d_t=w(pt)-w(qt)$.
  By \eqref{eq:k5-row-sum}, $\sum_{t\in A}d_t=0$ for every four-set $A\subseteq V(K_{10})\setminus\{p,q\}$.
  Let $r,s\notin\{p,q\}$ be distinct, and choose a three-set $B$ disjoint from $\{p,q,r,s\}$.
  Applying \eqref{eq:k5-row-sum} to $B\cup\{r\}$ and $B\cup\{s\}$ gives
  \[
    d_r+\sum_{b\in B}d_b=0,
    \text{ and }d_s+\sum_{b\in B}d_b=0.
  \]
  Hence $d_r=d_s$.
  Since $r$ and $s$ are arbitrary, all values $d_t$ are equal; denote their common value by $d$.
  Any four-term sum is then $4d=d$ in $\Z_3$, and therefore $d=0$.
  We have proved
  \begin{equation}
    w(pt)=w(qt)
    \label{eq:k5-row-equal}
  \end{equation}
  for $t\ne p,q$.

  Fix a vertex $t$.
  For any distinct $p,q\in V(K_{10})\setminus\{t\}$, \eqref{eq:k5-row-equal} gives $w(tp)=w(tq)$.
  Thus all edges incident with $t$ have the same weight.
  If $s\ne t$, the common weights of the edges incident with $s$ and with $t$ are both equal to $w(st)$, and hence are equal.
  Therefore, $w$ is constant, contradicting the fact that both values $0$ and $1$ occur.
  The contradiction proves the lemma.
\end{proof}

\section{The bound for $K_{2,2,2}$}
\label{app:k222-bound}

The second exceptional graph $K_{2,2,2}$ requires a different argument.
After fixing one part of the tripartition, the weight of a copy of $K_{2,2,2}$ can be encoded by the weight of a $4$-cycle under an auxiliary weighting.

We first prove a lemma about weightings of $K_9$ with no zero-sum $4$-cycle.

\begin{lemma}
  \label{lem:c4-rigidity}
  Let $w:E(K_9)\to\Z_3$ be a weighting with no zero-sum $4$-cycle.
  Then there exist an eight-set $U\subseteq V(K_9)$ and a residue $\varepsilon\in\Z_3^\times$ such that $w(uv)=\varepsilon$ for every distinct $u,v\in U$.
\end{lemma}

\begin{proof}
  For distinct vertices $u,v$ and $x\notin\{u,v\}$, define $f_{uv}(x)=w(ux)+w(vx)$.
  Since the $4$-cycle $uxvyu$ has weight $f_{uv}(x)+f_{uv}(y)$, the value $0$ occurs at most once among the values of $f_{uv}$, and the values $1$ and $2$ cannot both occur.
  Hence, for every pair $u,v$, there is a residue $q_{uv}\in\Z_3^\times$ such that $f_{uv}(x)=q_{uv}$ for all but at most one vertex $x\notin\{u,v\}$.
  When it exists, the remaining vertex has $f_{uv}(x)=0$ and will be called \emph{exceptional} for $uv$.

  We first find a six-set on which $w$ is constant.
  Fix a vertex $r$.
  For $x\ne r$, write $a_x=w(rx)$ and $c_x=a_x+q_{rx}$.
  If $c_x\ne c_y$, then either $y$ is exceptional for $rx$ or $x$ is exceptional for $ry$; otherwise,
  \[
    w(xy)=q_{rx}-a_y=q_{ry}-a_x
  \]
  would give $c_x=c_y$.
  For each $x\ne r$, at most one vertex $y\notin\{r,x\}$ is exceptional for the pair $rx$, so at most eight pairs $\{x,y\}\subseteq V(K_9)\setminus\{r\}$ satisfy $c_x\ne c_y$.
  If every residue occurred at most six times among the eight values $c_x$, then the number of unordered pairs $\{x,y\}$ with $c_x\ne c_y$ would be at least $6\cdot2=12$, with equality only for a $6+2$ split.
  Thus, there are a residue $c\in\Z_3$ and a set $X\subseteq V(K_9)\setminus\{r\}$ with $|X|\ge7$ such that $c_x=c$ for every $x\in X$.

  For $x\in X$, one has $q_{rx}=c-a_x\ne0$, and hence $a_x\ne c$.
  Moreover, for distinct $x,y\in X$, if $y$ is not exceptional for $rx$, then
  \begin{equation}
    w(xy)=c-a_x-a_y.
    \label{eq:c4-normal-form}
  \end{equation}
  The same formula holds if $x$ is not exceptional for $ry$.
  Exceptionality is mutual on $X$.
  Indeed, suppose that $y$ is exceptional for $rx$, but $x$ is not exceptional for $ry$.
  Then
  \[
    w(xy)=-a_y=q_{ry}-a_x=c-a_y-a_x,
  \]
  which forces $a_x=c$, contrary to $a_x\ne c$.
  Therefore, either \eqref{eq:c4-normal-form} holds or $x$ and $y$ are exceptional for each other.
  In the latter case, $w(xy)=-a_x=-a_y$, so $a_x=a_y$.
  Since each pair $rx$ has at most one exceptional vertex, these exceptional pairs form a matching $M$ on $X$.

  For every four-set $S\subseteq X$, the graph $M[S]$ is a matching of size at most two, and hence $K_S-M[S]$ contains a $4$-cycle.
  By \eqref{eq:c4-normal-form}, the weight of that cycle is
  \begin{equation}
    c+\sum_{x\in S}a_x\ne0.
    \label{eq:c4-fourset-sum}
  \end{equation}
  If $c=0$, then the labels $a_x$ lie in $\{1,2\}$.
  One value, say $\varepsilon$, occurs at least six times, while the other occurs at most once.
  Indeed, if both $1$ and $2$ occurred at least twice, we could choose a $4$-set $S\subseteq X$ containing two vertices of each label.
  Then $\sum_{x\in S}a_x=2\cdot1+2\cdot2=0$ in $\Z_3$, contradicting \eqref{eq:c4-fourset-sum} because $c=0$.
  Thus, one label occurs at most once, and the other occurs at least $|X|-1\ge6$ times.
  Suppose that a vertex $z\in X$ has the other label, so $a_z=-\varepsilon$.
  If $xy\in M$ with $a_x=a_y=\varepsilon$, choose an $\varepsilon$-labeled vertex $u\in X\setminus\{x,y\}$ such that $uz\notin M$; this is possible because $M$ is a matching and at least six vertices have label $\varepsilon$.
  Since $xy\in M$, while $yu,uz,zx\notin M$, we have
  \[
    w(xy)=-\varepsilon,\text{ }
    w(yu)=\varepsilon\text{ and }
    w(uz)=w(zx)=0.
  \]
  Thus, the $4$-cycle $xyuzx$ is zero-sum, a contradiction.
  Therefore, no two $\varepsilon$-labeled vertices form an edge of $M$.

  If every label equals $\varepsilon$, then $M$ has at most one edge, since two edges of $M$ together with the two cross edges form a zero-sum $4$-cycle.
  In either case, five $\varepsilon$-labeled vertices span no edge of $M$, and together with $r$ they form a constant $K_6$ of weight $\varepsilon$.

  Suppose instead that $c\ne0$.
  Then each $a_x$ is either $0$ or $-c$.
  Suppose that some $a_x=-c$.
  If at least three vertices of $X$ have label $0$, then one vertex labeled $-c$ together with three vertices labeled $0$ contradicts \eqref{eq:c4-fourset-sum}.
  Otherwise, at most two vertices have label $0$, so at least five vertices have label $-c$, and four of them again contradict \eqref{eq:c4-fourset-sum}.
  Hence, every $a_x$ is zero.
  The matching $M$ is empty: if $xy\in M$, choose $u,v\in X\setminus\{x,y\}$ with $uv\notin M$; then the $4$-cycle $xyuvx$ has weight $0+c+c+c=0$.
  Thus, every edge of $K_X$ has weight $c$, and any six vertices of $X$ form a constant $K_6$.

  Let $T$ be a constant six-set, and multiply all weights by the inverse of its nonzero edge weight.
  We may assume that every edge of $K_T$ has weight $1$.
  Write $T=\{1,\ldots,6\}$.
  For a vertex $z\notin T$ and $i\in T$, let $\alpha_i(z)=w(zi)$.
  A $4$-cycle through $z$ and three vertices of $T$ gives $\alpha_i(z)+\alpha_j(z)\ne1$ for $i\ne j$.
  This condition forbids the simultaneous occurrence of $0$ and $1$ among the coordinates and allows at most one coordinate equal to $2$.
  Consequently, $z$ has one of the following four types:
  \[
    A=(1,\ldots,1),
    \qquad
    A_i=A\text{ with the $i$-th entry replaced by }2,
  \]
  \[
    C=(0,\ldots,0),
    \qquad
    C_i=C\text{ with the $i$-th entry replaced by }2.
  \]

  Let $z,z'\notin T$, let their type vectors be $\alpha,\beta$, and set $\delta=w(zz')$.
  For distinct $i,j\in T$, the three $4$-cycles on $\{z,z',i,j\}$ imply
  \begin{equation}
    \alpha_i+\alpha_j+\beta_i+\beta_j\ne0, \text{ }
    \delta+\beta_i+1+\alpha_j \ne0 \text{ and }
    \delta+\beta_j+1+\alpha_i \ne0.
    \label{eq:c4-sixset-compatibility}
  \end{equation}
  Among pairs of non-$A$ types, the first inequality in \eqref{eq:c4-sixset-compatibility} leaves only a pair $A_p,C_p$ with the same index $p\in T$.
  Indeed, two types chosen from $C,C_1,\ldots,C_6$ have value $0$ in at least two common coordinates and therefore violate this inequality.
  The pair $A_p,C$ violates it at the coordinates $p$ and any $j\ne p$.
  The pair $A_p,A_q$ violates it at $p,q$ when $p\ne q$, and at $p$ and any $j\ne p$ when $p=q$.
  The pair $A_p,C_q$ with $p\ne q$ violates it at $p$ and any $k\notin\{p,q\}$.
  For the remaining pair $A_p,C_p$, however, one has $\alpha_t+\beta_t=1$ for every $t\in T$, so the first inequality is satisfied.

  Suppose that $z,z'$ have types $A_p,C_p$, respectively, and choose distinct $j,k\in T\setminus\{p\}$.
  Applying the last two inequalities in \eqref{eq:c4-sixset-compatibility} to $j,k$ gives $\delta+2\ne0$, and hence $\delta\ne1$.
  Applying them to $p,j$ gives $\delta+1\ne0$ and $\delta\ne0$, and hence $\delta\ne2,0$.
  Thus, the pair $A_p,C_p$ is also impossible, so among the three vertices outside $T$, at most one is not of type $A$.

  If all three external vertices have type $A$, every edge between them has weight $1$ or $2$.
  Two weight-$2$ edges would share a vertex and, together with any vertex of $T$, form a zero-sum $4$-cycle.
  Thus, two external vertices are joined by an edge of weight $1$, and adjoining them to $T$ gives a constant $K_8$.

  Finally, suppose that $z$ is the unique non-$A$ vertex and that $x,y$ have type $A$.
  The edge $xy$ has weight $1$ or $2$.
  Assume that $w(xy)=2$.
  Equation \eqref{eq:c4-sixset-compatibility} gives $w(zx)=w(zy)=1$ when $z$ has type $A_i$, $w(zx),w(zy)\in\{0,2\}$ when $z$ has type $C$, and $w(zx)=w(zy)=0$ when $z$ has type $C_i$.

  If $z$ has type $A_i$, the $4$-cycle $zxyiz$ has weight $1+2+1+2=0$.
  If $z$ has type $C$, the $4$-cycles $zixyz$ and $ziyxz$ force $w(zy)=w(zx)=2$, after which the $4$-cycle $zxiyz$ has weight zero.
  If $z$ has type $C_i$, choose $j\ne i$; then the $4$-cycle $zjxyz$ has weight $0+1+2+0=0$.
  Each case is impossible, so $w(xy)=1$ and $T\cup\{x,y\}$ is a constant $K_8$.

  This proves the lemma.
\end{proof}

\begin{proof}[Proof of \cref{lem:k222}]
  Suppose that a weighting $w:E(K_{11})\to\Z_3$ contains no zero-sum copy of $K_{2,2,2}$.
  Fix two vertices $a,b$, let $W=V(K_{11})\setminus\{a,b\}$, and for every $z\in W$, define $\lambda(z)=w(az)+w(bz)$.
  On $K_W$, set
  \[
    \omega(uv)=w(uv)-\lambda(u)-\lambda(v).
  \]

  For a $4$-cycle $C=uxvyu$ in $K_W$, let $H_C$ denote the copy of $K_{2,2,2}$ with vertex classes $\{a,b\}$, $\{u,v\}$, and $\{x,y\}$.
  Since each vertex of $C$ lies on two cycle edges,
  \[
    \omega(C)
    =w(C)-2\sum_{z\in V(C)}\lambda(z)
    =w(C)+\sum_{z\in V(C)}\lambda(z)
    =w(H_C).
  \]
  Here $-2=1$ in $\Z_3$.
  Hence, $\omega$ has no zero-sum $4$-cycle.

  Since $|W|=9$, applying \cref{lem:c4-rigidity} to $\omega$ gives an eight-set $U\subseteq W$ and a residue $\varepsilon\in\Z_3^\times$ such that $\omega(uv)=\varepsilon$ for every distinct $u,v\in U$.
  By the definition of $\omega$,
  \[
    w(uv)=\varepsilon+\lambda(u)+\lambda(v)
  \]
  for every distinct $u,v\in U$.

  Apply \cref{thm:EGZ} to any five vertices of $U$ to obtain a triple $X\subseteq U$ with $\sum_{x\in X}\lambda(x)=0$.
  Since $|U\setminus X|\ge5$, a second application gives a disjoint triple $Y\subseteq U\setminus X$ with $\sum_{y\in Y}\lambda(y)=0$.
  Partition $X\cup Y$ into three pairs, and let $H$ be the resulting copy of $K_{2,2,2}$.
  Since $K_{2,2,2}$ is $4$-regular and has $12$ edges,
  \[
    w(H)
    =12\varepsilon+4\sum_{z\in X\cup Y}\lambda(z)
    =0,
  \]
  a contradiction.
\end{proof}


\begin{thebibliography}{99}

  \bibitem{AlonCaro1993}
  N.~Alon and Y.~Caro,
  On three zero-sum Ramsey-type problems,
  \emph{J. Graph Theory} \textbf{17} (1993), 177--192.

  \bibitem{AlvaradoColucciParente2025}
  J.~D. Alvarado, L.~Colucci, and R.~Parente,
  On a problem of Caro on $\Z_3$-Ramsey number of forests,
  arXiv:2503.01032, 2025.

  \bibitem{ChungGraham1983}
F.~R.~K. Chung and R.~L. Graham,
Edge-colored complete graphs with precisely colored subgraphs,
\emph{Combinatorica} \textbf{3} (1983), no.~3--4, 315--324.

  \bibitem{BialostockiDierker1990}
  A.~Bialostocki and P.~Dierker,
  On zero sum Ramsey numbers: small graphs,
  \emph{Ars Combin.} \textbf{29A} (1990), 193--198.

  \bibitem{BialostockiDierker1990theorems}
  A.~Bialostocki and P.~Dierker,
  Zero sum Ramsey theorems,
  \emph{Congr. Numer.} \textbf{70} (1990), 119--130.

  \bibitem{Caro1994}
  Y.~Caro,
  A linear upper bound in zero-sum Ramsey theory,
  \emph{Internat. J. Math. Math. Sci.} \textbf{17} (1994), no.~3, 609--612.

  \bibitem{Caro1994binary}
  Y.~Caro,
  A complete characterization of the zero-sum (mod 2) Ramsey numbers,
  \emph{J. Combin. Theory Ser. A} \textbf{68} (1994), 205--211.

  \bibitem{Caro1996}
  Y.~Caro,
  Zero-sum problems---a survey,
  \emph{Discrete Math.} \textbf{152} (1996), 93--113.

  \bibitem{Caro1997}
  Y.~Caro,
  Binomial coefficients and zero-sum Ramsey numbers,
  \emph{J. Combin. Theory Ser. A} \textbf{80} (1997), 367--373.

  \bibitem{CaroPersonalCommunication2026}
  Y.~Caro,
  Personal communication, 5 August 2026.

  \bibitem{CaroMifsud2025}
  Y.~Caro and X.~Mifsud,
  On zero-sum Ramsey numbers modulo 3,
  arXiv:2502.03864, 2025.

  \bibitem{ChiHeCyclesWheels}
  C.~Chi and J.~He,
  On zero-sum Ramsey numbers of cycles and wheels,
  arXiv:2605.14954, 2026.

  \bibitem{ChiHeCompleteBipartite2026}
  C.~Chi and J.~He,
  On zero-sum Ramsey numbers of complete bipartite graphs,
  arXiv:2606.29216, 2026.

  \bibitem{ChiHeMaCompleteGraphs2026}
  C.~Chi, J.~He, and F.~Ma,
  A note on zero-sum Ramsey numbers of complete graphs,
  arXiv:2607.23954, 2026.

  \bibitem{CostaFriendship2026}
  S.~Costa,
  Exact zero-sum Ramsey numbers for friendship graphs, books, and triangular windmills,
  Zenodo, 2026,
  \url{https://doi.org/10.5281/zenodo.21896624}.

  \bibitem{CostaBoundedTreewidth2026}
  S.~Costa,
  Zero-sum Ramsey numbers of bounded-treewidth graphs,
  Zenodo, 2026,
  \url{https://doi.org/10.5281/zenodo.21896927}.

  \bibitem{ErdosGinzburgZiv1961}
  P.~Erd\H{o}s, A.~Ginzburg, and A.~Ziv,
  Theorem in additive number theory,
  \emph{Bull. Res. Council Israel} \textbf{10F} (1961), 41--43.

  \bibitem{Golumbic1978}
  M.~C. Golumbic,
  Trivially perfect graphs,
  \emph{Discrete Math.} \textbf{24} (1978), 105--107.

  \bibitem{HarborthPiepmeyer1994}
  H.~Harborth and L.~Piepmeyer,
  The zero-sum Ramsey numbers $R(K_4,\Z_3)$ and $R(K_6,\Z_3)$,
  \emph{Congr. Numer.} \textbf{101} (1994), 51--54.

  \bibitem{HarborthPiepmeyer1996}
  H.~Harborth and L.~Piepmeyer,
  Zero-sum Ramsey numbers modulo 3,
  \emph{J. Combin. Theory Ser. A} \textbf{75} (1996), 145--147.

  \bibitem{HeathSimmons2026}
  E.~Heath and A.~Simmons,
  A linear upper bound on the $\Z_p$-Ramsey number of graphs with
  sufficiently large $2$-packing,
  arXiv:2605.21817, 2026.

  \bibitem{KatzLianMalekshahianShapiro2025}
  J.~Katz, X.~Lian, A.~Malekshahian, and A.~Shapiro,
  A linear upper bound for zero-sum Ramsey numbers of bounded degree
  graphs,
  arXiv:2512.17790, 2025.

  \bibitem{Shapiro2026}
  A.~Shapiro,
  A linear upper bound on zero-sum Ramsey numbers of $d$-degenerate
  graphs in $\Z_p$,
  arXiv:2604.10864, 2026.
\end{thebibliography}
\end{document}